\documentclass[a4paper,11pt]{amsart}

\usepackage{amssymb,mathrsfs,cite}

\newtheorem{theorem}{Theorem}[section]
\newtheorem{lemma}[theorem]{Lemma}
\newtheorem{corollary}[theorem]{Corollary}
\newtheorem{proposition}[theorem]{Proposition}

\theoremstyle{definition}
\newtheorem{example}[theorem]{Example}
\newtheorem{remark}[theorem]{Remark}

\numberwithin{equation}{section}

\makeatletter
\@namedef{subjclassname@2020}{%
  \textup{2020} Mathematics Subject Classification}
\makeatother

\DeclareMathOperator{\RE}{Re}
\DeclareMathOperator{\IM}{Im}
\DeclareMathOperator{\Log}{Log}
\DeclareMathOperator{\supp}{supp}
\DeclareMathOperator{\sym}{sym}
\DeclareMathOperator{\meas}{meas}

\begin{document}

\title[Probability density functions attached to random Euler products]
{Probability density functions attached to random Euler products for automorphic $L$-functions}

\author[M. Mine]{Masahiro Mine}
\address{Faculty of Science and Technology\\ Sophia University\\ 7-1 Kioi-cho, Chiyoda-ku, Tokyo 102-8554, Japan}
\email{m-mine@sophia.ac.jp}

\date{}

\begin{abstract}
In this paper, we study the value-distributions of $L$-functions of holomorphic primitive cusp forms in the level aspect. 
We associate such automorphic $L$-functions with probabilistic models called the random Euler products. 
First, we prove the existence of probability density functions attached to the random Euler products. 
Then various mean values of automorphic $L$-functions are expressed as integrals involving the density functions. 
Moreover, we estimate the discrepancies between the distributions of values of automorphic $L$-functions and those of the random Euler products. 
\end{abstract}



\thanks{The work of this paper was supported by Grant-in-Aid for JSPS Research Fellows (Grant Number JP21J00529).}

\maketitle

\section{Introduction and statement of results}\label{sec1}
The study of the value-distributions of zeta and $L$-functions began with the work of Bohr and his collaborators in the early 20th century. 
Bohr--Jessen \cite{BohrJessen1930, BohrJessen1932} especially proved the existence of a limiting probability measure for the Riemann zeta-function $\zeta(\sigma+it)$ in $t$-aspect, that is, there exists the limit value
\begin{gather}\label{eq220002}
W_\sigma(R)
=\lim_{T \to\infty} \frac{1}{2T}
\meas\{ t \in[-T,T] \mid \log\zeta(\sigma+it) \in R \}
\end{gather}
for $\sigma>1/2$, where $R$ is any rectangle on $\mathbb{C}$ with edges parallel to the axes. 
Their method was later refined by many authors including Jessen--Wintner \cite{JessenWintner1935} and Borchsenius--Jessen \cite{BorchseniusJessen1948}. 
On the other hand, Chowla--Erd\H{o}s \cite{ChowlaErdos1951} obtained an analogous result for Dirichlet $L$-functions $L(s,\chi)$ in $\chi$-aspect. 
They considered $L$-functions of real characters $\chi_d$ attached to positive discriminants $d$ and proved the existence of the limit value
\begin{gather}\label{eq220003}
G_\sigma(a)
=\lim_{X \to\infty} \frac{2}{X}
\#\{ 0<d \leq X \mid \text{ $d \equiv0,1 (\bmod 4)$ square-free, ${L}(\sigma,\chi_d) \leq a$ }\}
\end{gather}
for $\sigma>3/4$, where $a$ is any positive real number. 
Similar results for $L$-functions of degree one were proved by Barban \cite{Barban1966}, Elliott \cite{Elliott1970, Elliott1971, Elliott1972, Elliott1973}, Stankus \cite{Stankus1975a, Stankus1975b}, and others. 

The above limit values $W_\sigma(R)$ and $G_\sigma(a)$ can be regarded as the probability distributions of certain random variables. 
For example, let $X=(X_p)$ be a sequence of independent random variables indexed by prime numbers which are uniformly distributed on the unit circle. 
Then we define the random Euler product
\[
\zeta(\sigma,X)
=\prod_{p} \left( 1- \frac{X_p}{p^\sigma} \right)^{-1}
\]
for $\sigma>1/2$, where $p$ runs through all prime numbers. 
In this case, the limit value $W_\sigma(R)$ of \eqref{eq220002} is represented as 
\[
\mathbb{P}( \log{\zeta}(\sigma,X) \in R)
=W_\sigma(R),
\]
where $\mathbb{P}(E)$ indicates the probability of an event $E$. 
One of the modern approaches to the study of the Riemann zeta-function is to compare the distribution of values $\zeta(\sigma+it)$ with the probability distribution of $\zeta(\sigma,X)$. 
Put 
\[
\mathbb{P}_T(\log{\zeta}(\sigma+it) \in R)
=\frac{1}{2T} \meas \{ t \in [-T,T] \mid \log{\zeta}(\sigma+it) \in R \}. 
\]
Lamzouri--Lester--Radziwi{\l\l} \cite{LamzouriLesterRadziwill2019} refined formula \eqref{eq220002} by estimating the discrepancy between $\zeta(\sigma+it)$ and $\zeta(\sigma,X)$. 
Indeed, they proved the upper bound
\begin{align}\label{eq231327}
D_\sigma(T)
&:=\sup_{R} \left| \mathbb{P}_T( \log{\zeta}(\sigma,X) \in R)
-\mathbb{P}( \log{\zeta}(\sigma,X) \in R) \right| \nonumber\\
&\ll
\begin{cases}
(\log{T})^{-1} (\log\log{T})^2 
& \text{for $\sigma=1$}, \\
(\log{T})^{-\sigma} 
& \text{for $1/2<\sigma<1$},
\end{cases}
\end{align}
which improves further the result of Harman--Matsumoto \cite{HarmanMatsumoto1994}. 
One can obtain an analogous result for Dirichlet $L$-functions $L(s,\chi)$ in $\chi$-aspect without major changes from the method of \cite{LamzouriLesterRadziwill2019}. 

Except for the case in $t$-aspect, the treatment of an $L$-function of higher degree should be rather complicated due to the lack of the completely multiplicativity of the Dirichlet coefficients. 
For the last two decades, many researchers confronted this difficulty and obtained one-dimensional results by concentrating on the values at real points $s=\sigma$ with $1/2<\sigma \leq 1$; see Section \ref{sec1.2}. 
The purposes of this paper is to study the one- or two-dimensional value-distributions of $L$-functions of holomorphic primitive cusp forms, which are typical examples of $L$-functions of degree two. 
The precise definitions of such automorphic $L$-functions are described in Section \ref{sec1.1}. 
Then, as an adequate random model for the value-distribution of the automorphic $L$-functions, we define the random Euler product $L(s,\Theta)$ as follows. 
Let $\Theta=(\Theta_p)$ be a sequence of independent $[0,\pi]$-valued random variables. 
We consider the infinite product
\begin{gather}\label{eq220316}
L(s,\Theta)
=\prod_{p} \left( 1- \frac{2 \cos \Theta_p}{p^s}+ \frac{1}{p^{2s}} \right)^{-1},  
\end{gather}
where $p$ runs through all prime numbers. 
For any fixed complex number $s=\sigma+it$ with $\sigma>1/2$, a sufficient condition to the convergence of \eqref{eq220316} almost surely is that every $\Theta_p$ satisfies
\begin{gather}\label{eq220319}
\mathbb{E}[\cos \Theta_p]
\ll_\epsilon p^{-\frac{1}{2}+\epsilon}
\end{gather}
for each $\epsilon>0$. 
Here, $\mathbb{E}[X]$ indicates the expected value of a random variable $X$. 
We present standard examples of the sequence $\Theta=(\Theta_p)$ in Example \ref{exaTF}, and we find that \eqref{eq220319} is just a technical condition to define $L(s,\Theta)$ for $\sigma>1/2$. 
The random Euler product $L(s,\Theta)$ is a $\mathbb{C}$-valued random variable in this case. 
We begin by showing the existence of a probability density function attached to $\log{L}(s,\Theta)$. 

\begin{theorem}\label{thmPDF}
Let $\Theta=(\Theta_p)$ be a sequence of independent $[0,\pi]$-valued random variables. 
Suppose that every $\Theta_p$ satisfies condition \eqref{eq220319} for each $\epsilon>0$ and 
\begin{gather}\label{eq220320}
\mathbb{E}[(\cos \Theta_p)^2]
\geq \delta
\end{gather}
with some absolute constant $\delta>0$. 
Then we have the following results. 
\begin{itemize}
\item[$(\mathrm{i})$]
Let $s=\sigma+it$ be a fixed complex number with $\sigma>1/2$ and $t \neq0$. 
Then there exists a continuous function $\mathcal{M}_s(\,\cdot\,,\Theta): \mathbb{C} \to \mathbb{R}_{\geq0}$ such that 
\[
\mathbb{P}( \log{L}(s,\Theta) \in A)
=\int_{A} \mathcal{M}_s(w,\Theta) \,|dw|
\]
holds for all $A \in \mathcal{B}(\mathbb{C})$, where we write $|dw|=(2\pi)^{-1} dudv$ for $w=u+iv$.  
\item[$(\mathrm{ii})$]
Let $\sigma>1/2$ be a fixed real number. 
Then there exists a continuous function $\mathcal{M}_\sigma(\,\cdot\,,\Theta): \mathbb{R} \to \mathbb{R}_{\geq0}$ such that 
\[
\mathbb{P}( \log{L}(\sigma,\Theta) \in A)
=\int_{A} \mathcal{M}_\sigma(u,\Theta) \,|du|
\]
holds for all $A \in \mathcal{B}(\mathbb{R})$, where we write $|du|=(2\pi)^{-1/2}du$. 
\end{itemize}
Here, we denote by $\mathcal{B}(S)$ the class of Borel sets of a topological space $S$. 
\end{theorem}

It is notable that the notation $\mathcal{M}_s(\,\cdot\,,\Theta)$ indicates a function on $\mathbb{C}$ if $s$ is not real, while the same notation indicates a function on $\mathbb{R}$ if $s$ is real. 
We obtain a relation between $\mathcal{M}_{\sigma+it}(w,\Theta)$ with $t \neq 0$ and $\mathcal{M}_\sigma(u,\Theta)$ as \eqref{eq300000} in Section \ref{sec4.1}. 
Note further that Bohr--Jessen \cite{BohrJessen1930, BohrJessen1932} obtained a continuous function $F_\sigma: \mathbb{C} \to \mathbb{R}_{\geq0}$ such that
\begin{gather}\label{eq240107}
W_\sigma(R)
=\int_{R} F_\sigma(u+iv) \,dudv
\end{gather}
for $\sigma>1/2$. 
Hence the density function $\mathcal{M}_s(\,\cdot\,,\Theta)$ is an analogue of $F_\sigma$ up to a multiplicative constant.

\subsection{The value-distributions of automorphic $L$-functions}\label{sec1.1}
Let $q$ be a prime number, and denote by $S_2(q)$ the vector space of holomorphic cusp forms for $\Gamma_0(q)$ of weight 2 with trivial nebentypus. 
For a cusp form $f \in S_2(q)$, we have the Fourier series expansion 
\[
f(z)
=\sum_{n=1}^{\infty} a(n,f) \sqrt{n} \exp(2\pi inz). 
\]
Then the automorphic $L$-function attached to $f \in S_2(q)$ is defined as 
\[
L(s,f)
=\sum_{n=1}^{\infty} \frac{a(n,f)}{n^s}. 
\]
Let $B_2(q)$ denote an orthogonal basis of $S_2(q)$ consisting of primitive cusp forms normalized so that $a(1,f)=1$. 
The coefficient $a(n,f)$ is multiplicative if $f \in B_2(q)$. 
By the achievement of Deligne, there exists a real number $\theta_p(f) \in [0,\pi]$ such that $a(p,f)=2\cos \theta_p(f)$ for $p \neq q$, and more generally, we have
\begin{gather}\label{eq070007}
a(p^m,f)
=
\begin{cases}
U_m(\cos \theta_p(f)) 
& \text{for $p \neq q$}, \\
a(q,f)^m 
& \text{for $p=q$},
\end{cases}
\end{gather}
where $U_m(x)$ is the $m$-th Chebyshev polynomial of the second kind. 
This implies that the automorphic $L$-function $L(s,f)$ for $f \in B_2(q)$ has the Euler product 
\begin{gather}\label{eq220310}
L(s,f)
=\left( 1- \frac{a(q,f)}{q^s} \right)^{-1}
\prod_{p \neq q} \left(1- \frac{a(p,f)}{p^s}+ \frac{1}{p^{2s}} \right)^{-1}  
\end{gather}
which is convergent absolutely for $\RE(s)>1$. 
Recall that $L(s,f)$ is analytically continued to the whole complex plane and that a usual functional equation holds between $L(s,f)$ and $L(1-s,f)$. 

Let $a(n,\Theta)$ be the random variable defined by extending $a(p^m,\Theta)=U_m(\cos \Theta_p)$ multiplicatively in $n$. 
Then the random variable $L(s,\Theta)$ of \eqref{eq220316} is represented as 
\[
L(s,\Theta)
=\sum_{n=1}^{\infty} \frac{a(n,\Theta)}{n^s} 
\]
for $\RE(s)>1/2$ almost surely if condition \eqref{eq220319} is satisfied. 
We further suppose that
\begin{gather}\label{eq220323}
\mathbb{E}[a(n,\Theta)]
=0
\end{gather}
unless $n$ is a perfect square. 
Note that \eqref{eq220323} implies \eqref{eq220319} by definition. 
We use \eqref{eq220323} to derive a large sieve inequality; see Proposition \ref{lemIM} and Corollary \ref{corIM}. 
To compare $L(s,f)$ and $L(s,\Theta)$, we consider the following relation between $a(n,f)$ and $a(n,\Theta)$: there exist absolute constants $\alpha, \beta>0$ such that 
\begin{gather}\label{eq220321}
\sum_{f \in B_2(q)} \omega_q(f) a(n,f)
=\mathbb{E}[a(n,\Theta)]
+O\left( n^\alpha q^{-\beta} \right) 
\end{gather}
for $(n,q)=1$, where $\omega_q$ is a non-negative weight function defined on $B_2(q)$.  
Remark that \eqref{eq220321} implies $\sum_{f \in B_2(q)} \omega_q(f) \to1$ as $q \to\infty$ since $a(1,f)=a(1,\Theta)=1$.  
For some technical reasons, we suppose that the weight function $\omega_q$ satisfies
\begin{gather}\label{eq220322}
\max_{f \in B_2(q)} \omega_q(f)
\ll q^{-1} (\log{q})^A 
\end{gather}
with some absolute constant $A \geq0$. 

\begin{example}\label{exaTF}
We have the following examples of $\Theta$ and $\omega_q$ satisfying conditions \eqref{eq220320}, \eqref{eq220323}, \eqref{eq220321}, and \eqref{eq220322}. 
\begin{itemize}
\item[$(\mathrm{1})$]
Let $\Theta^1=(\Theta_p^1)$ be a sequence of independent random variables identically distributed on $[0,\pi]$ according to the Sato--Tate measure
\[
d \mu_\infty(\theta)
=\frac{2}{\pi} \sin^2 \theta \,d \theta.
\] 
By definition, the expected values are calculated as 
\[
\mathbb{E}[(\cos \Theta_p^1)^2]
=1/4
\quad\text{and}\quad
\mathbb{E}[U_m(\cos \Theta_p^1)]
=0
\]
for any integer $m \geq1$. 
Hence conditions \eqref{eq220320} and \eqref{eq220323} are satisfied with $\Theta=\Theta^1$. 
Let $\langle f,f \rangle$ denote the Petersson norm of $f$. 
Then condition \eqref{eq220322} holds with the harmonic weight $\omega_q(f)=(4 \pi \langle f,f \rangle)^{-1}$ by the bounds 
\[
q(\log{q})^{-1}
\ll \langle f,f \rangle
\ll q(\log{q})^3;
\]
see \cite[$(1.17)$]{CogdellMichel2004}.  
Finally, the Petersson trace formula \cite[Proposition 1.9]{CogdellMichel2004} yields relation \eqref{eq220321} in this case.  
\item[$(\mathrm{2})$]
Let $\Theta^2=(\Theta_p^2)$ be a sequence of independent random variables, and let every $\Theta_p^2$ be distributed on $[0,\pi]$ according to the $p$-adic Plancherel measure
\[
d \mu_p(\theta)
=\left(1+\frac{1}{p}\right) \left( 1- \frac{2\cos 2\theta}{p}+ \frac{1}{p^2} \right)^{-1}
d \mu_\infty(\theta).
\] 
Then conditions \eqref{eq220320} and \eqref{eq220323} are satisfied with $\Theta=\Theta^2$ since 
\[
\mathbb{E}[(\cos \Theta_p^2)^2]
=(1+p^{-1})/4
\quad\text{and}\quad 
\mathbb{E}[U_m(\cos \Theta_p^2)]=0
\]
for any odd integer $m \geq1$. 
We also recall that the formula 
\[
\# B_2(q)
=\dim S_2(q)
=\frac{q}{12}+O(1)
\]
holds. 
Hence \eqref{eq220322} holds with $\omega_q(f)=\# B_2(q)^{-1}$. 
Moreover, relation \eqref{eq220321} is derived from the Eichler--Selberg trace formula; see \cite[Proposition 2.8]{Brumer1995}. 
\end{itemize}
\end{example}

The Grand Riemann Hypothesis (GRH) asserts that all non-trivial zeros of $L(s,f)$ lie on the critical line $\RE(s)=1/2$. 
However, we shall permit the existence of zeros off the critical line. 
Hence $\log{L}(s,f)$ may not be defined as a holomorphic function on the right-half plane $D=\{ s \in \mathbb{C}\mid \RE(s)>1/2 \}$. 
We fix the branch of $\log{L}(s,f)$ as follows. 
First, we define $\log{L}(s,f)$ for $\RE(s)>1$ by 
\[
\log{L}(s,f)
=\sum_{p} \sum_{m=1}^{\infty} \frac{b(p^m,f)}{p^{ms}}
\]
according to Euler product representation \eqref{eq220310}, where $b(p^m,f)$ is given by
\begin{gather}\label{eq052248}
b(p^m,f)
=
\begin{cases}
2\cos(m\theta_p(f))/m 
& \text{for $p \neq q$}, \\
a(q,f)^m/m 
& \text{for $p=q$}.
\end{cases}
\end{gather}
We extend $\log{L}(s,f)$ for $s \in G_f$ by the analytic continuation along the horizontal path from right, where $G_f$ is the set
\begin{gather}\label{eq220330}
G_f
=D \setminus \bigcup_{\substack{ L(\rho,f)=0 \\ \RE(\rho)>1/2 }}
\{ \sigma+i\IM(\rho) \mid 1/2<\sigma \leq \RE(\rho) \}. 
\end{gather}

Let $S=\mathbb{R}$ or $\mathbb{C}$. 
Denote by $C(S)$ the class of all continuous functions on $S$.  
Then we define three subclasses of $C(S)$ as 
\begin{align*}
{C}^{\exp}(S)
&=\left\{ \Phi \in{C}(S) ~\middle|~ \text{$\Phi(x) \ll e^{a|x|}$ with some $a>0$} \right\}, \\
{C}^{\mathrm{poly}}(S)
&=\left\{ \Phi \in{C}(S) ~\middle|~ \text{$\Phi(x) \ll |x|^a$ with some $a>0$} \right\}, \\
{C}_b(S)
&=\left\{ \Phi \in{C}(S) ~\middle|~ \text{$\Phi$ is bounded} \right\}. 
\end{align*}
We further define
\begin{gather*}
\mathcal{I}(S)
=\left\{ 1_A ~\middle|~ \text{$A$ is a continuity set of $S$} \right\}, 
\end{gather*}
where $1_A$ is the indicator function of a set $A \subset S$, and a Borel set $A$ is called a continuity set of $S$ if its boundary $\partial A$ has Lebesgue measure zero in $S$. 
Let
\[
B'_2(q,s)
=\{ f \in B_2(q) \mid s \in G_f \},
\]
where $G_f$ is the set of \eqref{eq220330}. 
Then $B'_2(q,s)=B_2(q)$ holds for $\RE(s) \geq1$ unconditionally since there exist no zeros of $L(s,f)$ with $\RE(\rho) \geq1$. 
Similarly, we have $B'_2(q,s)=B_2(q)$ for any $\RE(s)>1/2$ under the assumption of GRH. 

\begin{theorem}\label{thmMV}
Let $\Theta=(\Theta_p)$ be a sequence of independent $[0,\pi]$-valued random variables and $\omega_q$ be a non-negative weight function on $B_2(q)$. 
Suppose that $\Theta$ and $\omega_q$ satisfies \eqref{eq220320}, \eqref{eq220323}, \eqref{eq220321}, and \eqref{eq220322}. 
Then we have the following results. 
\begin{itemize}
\item[$(\mathrm{i})$]
Let $s=\sigma+it$ be a fixed complex number with $\sigma>1/2$ and $t \neq 0$. 
Then the limit formula
\begin{gather}\label{eq111435}
\lim_{\substack{ q \to\infty \\ \text{{\rm $q$: prime}} }}
\sum_{f \in B'_2(q,s)} \omega_q(f) \Phi(\log{L}(s,f))
=\int_{\mathbb{C}} \Phi(w) \mathcal{M}_s(w,\Theta) \,|dw|
\end{gather}
holds in the following cases: 
\begin{itemize}
\item 
$\sigma>1$ and $\Phi \in C(\mathbb{C})\cup\mathcal{I}(\mathbb{C})$; 
\item 
$\sigma=1$ and $\Phi \in C^{\mathrm{poly}}(\mathbb{C})\cup\mathcal{I}(\mathbb{C})$; 
\item 
$1/2<\sigma<1$ and $\Phi \in C_b(\mathbb{C})\cup\mathcal{I}(\mathbb{C})$ without assuming GRH; 
\item 
$1/2<\sigma \leq1$ and $\Phi \in C^{\exp}(\mathbb{C})\cup\mathcal{I}(\mathbb{C})$ if we assume GRH.  
\end{itemize}
\item[$(\mathrm{ii})$]
Let $\sigma>1/2$ be a fixed real number. 
Then the limit formula
\[
\lim_{\substack{ q \to\infty \\ \text{{\rm $q$: prime}} }}
\sum_{f \in B'_2(q,\sigma)} \omega_q(f) \Phi(\log{L}(\sigma,f))
=\int_{\mathbb{R}} \Phi(u) \mathcal{M}_\sigma(u,\Theta) \,|du|
\]
holds in the following cases: 
\begin{itemize}
\item 
$\sigma>1$ and $\Phi \in C(\mathbb{R})\cup\mathcal{I}(\mathbb{R})$; 
\item 
$\sigma=1$ and $\Phi \in C^{\exp}(\mathbb{R})\cup\mathcal{I}(\mathbb{R})$; 
\item 
$1/2<\sigma<1$ and $\Phi \in C_b(\mathbb{R})\cup\mathcal{I}(\mathbb{R})$ without assuming GRH; 
\item 
$1/2<\sigma<1$ and $\Phi \in C^{\exp}(\mathbb{R})\cup\mathcal{I}(\mathbb{R})$ if we assume GRH.  
\end{itemize}
\end{itemize}
\end{theorem}

Theorem \ref{thmMV} presents an analogue of limit formulas \eqref{eq220002} and \eqref{eq220003} for automorphic $L$-functions if the test function $\Phi$ is taken from $\mathcal{I}(S)$. 
Moreover, we calculate the complex moments of $L(s,f)$ as
\[
\lim_{\substack{ q \to\infty \\ \text{$q$: prime} }}
\sum_{ f \in B_2(q)} \omega_q(f) \overline{L(s,f)}^{z} L(s,f)^{z'}
=\int_{\mathbb{C}} e^{z\overline{w}+z'w} \mathcal{M}_s(w,\Theta) \,|dw| 
\]
for any $z,z' \in \mathbb{C}$ in the case $s=\sigma+it \notin \mathbb{C}$, where $\sigma>1$ unconditionally; and $\sigma>1/2$ if we assume GRH. 
Similarly, we obtain
\[
\lim_{\substack{ q \to\infty \\ \text{$q$: prime} }}
\sum_{f \in B_2(q)} \omega_q(f) \overline{\log{L}(s,f)}^{z} (\log{L}(s,f))^{z'}
=\int_{\mathbb{C}} \overline{w}^z w^{z'} \mathcal{M}_s(w,\Theta) \,|dw| 
\]
where $\sigma \geq1$ unconditionally; and $\sigma>1/2$ under GRH. 
A similar result was found by Ihara--Murty--Shimura \cite{IharaMurtyShimura2009} for the logarithmic derivatives of Dirichlet $L$-functions. 
Then, we derive discrepancy estimates similar to \eqref{eq231327}. 
We define 
\begin{gather}\label{eq290016}
D_s(q,\Theta,\omega_q)
=\sup_{R \subset \mathbb{C}}
\Bigg| \sum_{f \in B'_2(q,s)} \omega_q(f) 1_R\left(\log{L}(s,f)\right)
-\int_{R} \mathcal{M}_s(w,\Theta) \,|dw| \Bigg|
\end{gather}
if $s=\sigma+it \in \mathbb{C}$ with $\sigma>1/2$ and $t \neq 0$, where $R$ runs through all rectangles on $\mathbb{C}$ with sides parallel to the axes. 
Similarly, we define
\begin{gather}\label{eq290017}
D_\sigma(q,\Theta,\omega_q)
=\sup_{[a,b] \subset \mathbb{R}}
\Bigg| \sum_{f \in B'_2(q,\sigma)} \omega_q(f) 1_{[a,b]}\left(\log{L}(\sigma,f)\right)
-\int_{a}^{b} \mathcal{M}_\sigma(u,\Theta) \,|du| \Bigg|
\end{gather}
if $s=\sigma>1/2$. 
Remark that we use the same notation $D_s(q,\Theta,\omega_q)$ to indicate discrepancies \eqref{eq290016} and \eqref{eq290017} for simplicity. 

\begin{theorem}\label{thmDB}
Let $\Theta$ and $\omega_q$ be as in Theorem \ref{thmMV}. 
Then we obtain
\[
D_s(q,\Theta,\omega_q)
\ll
\begin{cases}
(\log{q})^{-1} \log\log{q} 
& \text{if $\sigma>1$, } \\
(\log{q})^{-1} \log\log{q} \log\log\log{q} 
& \text{if $\sigma=1$, } \\
(\log{q})^{-\sigma} 
& \text{if $1/2<\sigma<1$ }
\end{cases}
\]
in both cases of $s=\sigma+it \notin \mathbb{R}$ and $s=\sigma \in \mathbb{R}$. 
Here, the implied constants depend only on $s$ and the choice of $\Theta$. 
\end{theorem}

\subsection{Notes on related results}\label{sec1.2}
The method of extending the result of Bohr--Jessen \cite{BohrJessen1930, BohrJessen1932} or Chowla--Erd\H{o}s \cite{ChowlaErdos1951} for higher degree $L$-functions has been developed by many researchers. 
Regarding the results in $t$-aspect, Matsumoto obtained analogues of formula \eqref{eq220002} for $L$-functions of holomorphic primitive cusp forms \cite{Matsumoto1989}, Dedekind zeta-functions \cite{Matsumoto1991, Matsumoto1992}, and more generally, a class of $L$-functions with polynomial Euler products \cite{Matsumoto1990}. 
Except for the case in $t$-aspect, the study of the value-distributions of $L$-functions of high degrees was initiated by Luo \cite{Luo1999}, who considered symmetric square $L$-functions of Maass forms in the aspect of Laplacian eigenvalues. 
He succeeded to show an analogue of \eqref{eq220003} in this case. 
Royer \cite{Royer2001} and Fomenko \cite{Fomenko2003} proved similar results with holomorphic cusp forms. 
Note that their methods relies on the calculations of all integral moments such as  
\[
\lim_{\substack{ q \to\infty \\ \text{$q$: prime} }}
\frac{1}{\#{B}_2(q)} \sum_{f \in B_2(q)} L(1,f)^m
=B_m,  
\quad m\in\mathbb{Z}_{\geq1}
\]
with some constants $B_m>0$. 
Hence the results were restricted to the values at $s=1$ and did not yields the discrepancy bounds. 
Cogdell--Michel \cite{CogdellMichel2004}, Golubeva \cite{Golubeva2004}, and Fomenko \cite{Fomenko2005} calculated complex moments of $L$-functions of holomorphic cusp forms. 
As applications of the results, they estimated 
\begin{gather}\label{eq231702}
D_1(q,\Theta,\omega_q)
\ll (\log{q})^{-1} \log\log{q} \log\log\log{q}
\end{gather}
with $\Theta=\Theta^1$, $\omega_q=(4 \pi \langle f,f \rangle)^{-1}$ in \cite[Corollary 1.16]{CogdellMichel2004}; and 
\begin{gather}\label{eq231703}
D_1(q,\Theta,\omega_q)
\ll (\log{q})^{-1/2} \log\log{q}
\end{gather}
with $\Theta=\Theta^2$, $\omega_q=\# B_2(q)^{-1}$ in \cite[Theorem 1]{Golubeva2004}. 
Theorem \ref{thmDB} generalizes and improves discrepancy estimates \eqref{eq231702} and \eqref{eq231703}. 

Next, we recall the results of ``$M$-functions'' which are density functions closely related to the value-distributions of $L$-functions. 
Ihara--Matsumoto \cite{IharaMatsumoto2011a, IharaMatsumoto2011b} proved that various mean values of $\zeta(\sigma+it)$ in $t$-aspect or $L(s,\chi)$ in $\chi$-aspect are represented as integrals involving the associated $M$-functions. 
Inspired by their work, several researchers studied certain mean values of automorphic $L$-functions in the level aspect. 
Then, attempts to obtain analogues of the results proved in \cite{IharaMatsumoto2011a, IharaMatsumoto2011b} were partially succeeded. 
Actually, Matsumoto--Umegaki \cite{MatsumotoUmegaki2018} proved the formula
\begin{align*}
&\lim_{\substack{ q \to\infty \\ \text{$q$: prime} }}
\sum_{ f \in B_k(q)} \frac{\Gamma(k-1)}{(4\pi)^{k-1} \langle f,f \rangle}
\Phi\left( \log{L}_{\mathbb{P}(q)}(\sigma,\sym^{\mu} f)
-\log{L}_{\mathbb{P}(q)}(\sigma,\sym^{\nu} f) \right)\\
&\qquad
=\int_{\mathbb{R}} \Phi(u) \mathcal{M}_\sigma(u) \,|du|
\end{align*}
with $\mu=\nu+2 \geq3$, where ${L}_{\mathbb{P}(q)}(s,\sym^{\gamma} f)$ is a slightly modified symmetric power $L$-function of $f$. 
The test function $\Phi$ is taken from ${C}_b(\mathbb{R})\cup\mathcal{I}(\mathbb{R})$. 
In this result, the $M$-function $\mathcal{M}_\sigma$ satisfies
\[
\frac{1}{\sqrt{2\pi}} \mathcal{M}_\sigma(u)
=\int_{-\infty}^{\infty} F_\sigma(u+iv) \,dv, 
\]
where $F_\sigma$ is the density function of \eqref{eq240107}. 
They also considered the case $f$ runs over the set $B_k(q^m)$ with $m \geq1$; see \cite[Theorem 1.5]{MatsumotoUmegaki2018} for the precise statement. 
Another result was obtained by Lebacque--Zykin \cite{LebacqueZykin2018}. 
They proved
\begin{align}\label{eq240135}
&\lim_{\substack{ q \to\infty \\ \text{$q$: prime} }}
\sum_{f \in B_k(q)} \frac{\Gamma(k-1)}{(4\pi)^{k-1} \langle f,f \rangle}
\overline{L(s,f)^{\frac{iz}{2}}} L(s,f)^{\frac{iz'}{2}} \\
&\qquad
=\sum_{m,n \in \mathbb{N}} m^{-\overline{s}} n^{-s} \sum_{x \in J(m) \cap J(n)}
\overline{c_{z,x}(m)} c_{z',x}(n) \nonumber
\end{align}
for $z,z' \in \mathbb{C}$ under GRH, where the coefficient $c_{z,x}(n)$ and the subset $J(n) \subset \mathbb{N}$ can be explicitly calculated. 
However, they did not find a suitable $M$-function due to the complexity of the right-hand side of \eqref{eq240135}. 
In this paper, we construct the $M$-function $\mathcal{M}_\sigma(\,\cdot\,,\Theta)$ from the random Euler product $L(s,\Theta)$. 
Then we establish a better analogue of the results of Ihara--Matsumoto \cite{IharaMatsumoto2011a, IharaMatsumoto2011b} as Theorem \ref{thmMV}.

\vspace{\baselineskip}

This paper consists of four sections. 
First, we introduce Propositions \ref{propUB} and \ref{propCM1} which are key tools for the proofs of the results of Section \ref{sec1}. 
In Section \ref{sec2}, we prove Proposition \ref{propUB} and explain how Proposition \ref{propCM1} is proved. 
Then we approximate $\log{L}(s,f)$ and $\log{L}(s,\Theta)$ by Dirichlet polynomials. 
The proof of Proposition \ref{propCM1} is completed in Section \ref{sec3}. 
In Section \ref{sec4}, we deduce Theorems \ref{thmPDF}, \ref{thmMV}, and \ref{thmDB} from the key propositions. 
The methods are based on the results of Probability Theory, e.g.\ L\'{e}vy's inversion formula, the continuity theorem, and Esseen's inequality.

\section{Key propositions}\label{sec2}
According to Ihara--Matsumoto \cite{IharaMatsumoto2011b}, we introduce a quasi-character $\psi_{z,z'}$ of the additive group $\mathbb{C}$ such that
\[
\psi_{z,z'}(w)
=\exp\left( \frac{i}{2}(z\overline{w}+z'w) \right) 
\]
for $z,z'\in\mathbb{C}$. 
Then we begin with studying the function
\[
\widetilde{\mathcal{M}}_s(z,z',\Theta)
=\mathbb{E}\left[\psi_{z,z'}(\log{L}(s,\Theta))\right]. 
\]
We show that $\widetilde{\mathcal{M}}_s(z,z',\Theta)$ is defined for any $\RE(s)>1/2$ if $\Theta=(\Theta_p)$ satisfies condition \eqref{eq220319}; see Lemmas \ref{lemMGF} and \ref{lemIP}. 
In particular, the function 
\[
\widetilde{\mathcal{M}}_s(z,\overline{z},\Theta)
=\mathbb{E}\left[ \exp( i \RE{z} \RE\log{L}(s,\Theta)
+i \IM{z} \IM\log{L}(s,\Theta) ) \right]
\]
is the characteristic function of the random variable $\log{L}(s,\Theta)$. 
Remark that, if $s=\sigma \in \mathbb{R}$, then $\widetilde{\mathcal{M}}_\sigma(z,\overline{z},\Theta)$ does not depend on the imaginary part of $z$. 
Indeed, it is represented as
\[
\widetilde{\mathcal{M}}_\sigma(z,\overline{z},\Theta)
=\mathbb{E}\left[\exp(ix \log{L}(\sigma,\Theta))\right] 
\]
for $z=x+iy$. 
The first key step is to estimate the decay of the characteristic function $\widetilde{\mathcal{M}}_s(z,\overline{z},\Theta)$ as $|z|\to\infty$. 

\begin{proposition}\label{propUB}
Suppose that $\Theta=(\Theta_p)$ satisfies conditions \eqref{eq220319} and \eqref{eq220320}. 
Then we have the following results. 
\begin{itemize}
\item[$(\mathrm{i})$]
Let $s=\sigma+it$ be a fixed complex number with $\sigma>1/2$ and $t \neq 0$. 
There exists a positive constant $c(s)$ such that the inequality
\[
\left| \widetilde{\mathcal{M}}_s(z,\overline{z},\Theta) \right|
\leq\exp\left( -c(s) \frac{|z|^{1/\sigma}}{\log|z|} \right)
\]
holds for all $z=x+iy \in \mathbb{C}$ with $|z|\geq3$. 
\item[$(\mathrm{ii})$]
Let $\sigma>1/2$ be a fixed real number. 
There exists a positive constant $c(\sigma)$ such that the inequality
\[
\left| \widetilde{\mathcal{M}}_\sigma(z,\overline{z},\Theta) \right|
\leq\exp\left( -c(\sigma) \frac{|x|^{1/\sigma}}{\log|x|} \right)
\]
holds for all $z=x+iy \in \mathbb{C}$ with $|x|\geq3$. 
\end{itemize}
\end{proposition}

As the second key step, we associate the mean value of $\psi_{z,z'}(\log{L}(s,f))$ with the function $\widetilde{\mathcal{M}}_s(z,z',\Theta)$ for $z,z'\in\mathbb{C}$ when they satisfy $z'=\pm z$ or $\overline{z}$. 
According to the method of Lamzouri--Lester--Radziwi{\l\l} \cite{LamzouriLesterRadziwill2019}, we consider the mean value 
\[
\widetilde{\mathcal{M}}_{s,q}(z,z',\omega_q)^{\mathcal{E}}
=\sum_{f \in B_2(q) \setminus \mathcal{E}} \omega_q(f) \psi_{z,z'}(\log{L}(s,f))
\]
with an exceptional subset $\mathcal{E}=\mathcal{E}(q,s)$ of $B_2(q)$. 

\begin{proposition}\label{propCM1}
Let $s=\sigma+it$ be a fixed complex number with $\sigma>1/2$, and let $B \geq1$. 
Suppose that $\Theta=(\Theta_p)$ and $\omega_q$ satisfies conditions \eqref{eq220323}, \eqref{eq220321}, and \eqref{eq220322}. 
Then there exist positive constants $a_1=a_1(\sigma,B)$ and $b_1=b_1(\sigma,B)$ such that 
\begin{gather}\label{eq302035}
\widetilde{\mathcal{M}}_{s,q}(z,z',\omega_q)^{\mathcal{E}}
=\widetilde{\mathcal{M}}_s(z,z',\Theta)
+O\left(\frac{1+E_s(z,z'\Theta)}{(\log{q})^B}\right) 
\end{gather}
for all $z,z'\in\mathbb{C}$ satisfying $z'=\pm z$ or $\overline{z}$ in the range $|z| \leq a_1R_\sigma(q)$, where we put 
\begin{align}
&E_s(z,z',\Theta)
=\mathbb{E}\left[|\psi_{z,z'}(\log{L}(s,\Theta))|\right], \label{eq301947} \\
&R_\sigma(q) 
=
\begin{cases}
(\log{q}) (\log\log{q})^{-1} 
&{\text{if $\sigma>1$}}, \\
(\log{q}) (\log\log{q} \log\log\log{q})^{-1} 
&{\text{if $\sigma=1$}}, \\
(\log{q})^\sigma 
&{\text{if $1/2<\sigma<1$}}. 
\end{cases} \label{eq301946}
\end{align}
Here, $\mathcal{E}=\mathcal{E}(q,s)$ is a subset of $B_2(q)$ satisfying
\begin{gather}\label{eq302036}
\frac{\#\mathcal{E}(q,s)}{q}
\ll \exp\left( -b_1 \frac{\log{q}}{\log\log{q}} \right). 
\end{gather}
The implied constants in \eqref{eq302035} and \eqref{eq302036} depend only on $s$ and $B$. 
\end{proposition}

\subsection{Proof of Proposition \ref{propUB}}\label{sec2.1}

\begin{lemma}\label{lemMGF}
Let $s=\sigma+it$ be a fixed complex number with $\sigma>1/2$. 
Suppose that $\Theta=(\Theta_p)$ satisfies condition \eqref{eq220319}. 
Then the expected values
\[
\mathbb{E}[\exp( a |\RE\log{L}(s,\Theta)| )]
\quad\text{and}\quad
\mathbb{E}[\exp( a |\IM\log{L}(s,\Theta)| )]
\]
are finite for any fixed real number $a \geq0$. 
\end{lemma}

\begin{proof}
To begin with, we show that the expected value $\mathbb{E}[\exp( b \RE\log{L}(s,\Theta)| )]$ is finite for any $b \in \mathbb{R}$. 
Let $L_Y(s,\Theta)$ be the partial Euler product
\begin{gather}\label{eq110122}
L_Y(s,\Theta)
=\prod_{p \leq Y} \left( 1- \frac{2\cos \Theta_p}{p^s}+ \frac{1}{p^{2s}} \right)^{-1}. 
\end{gather}
Since $L_Y(s,\Theta)$ converges to $L(s,\Theta)$ as $Y \to\infty$ almost surely, we obtain 
\begin{gather}\label{eq251719}
\mathbb{E}[\exp( b \RE \log{L}(s,\Theta) )]
\leq \liminf_{Y \to\infty} \mathbb{E}[\exp( b \RE \log{L}_Y(s,\Theta) )]
\end{gather}
by Fatou's lemma. 
From the independence of $\Theta=(\Theta_p)$, we deduce 
\[
\mathbb{E}[\exp( b \RE \log{L}_Y(s,\Theta) )]
=\prod_{p \leq Y}
\mathbb{E}\left[ \exp\left(-b \RE \log(1- 2(\cos \Theta_p)p^{-s}+ p^{-2s})\right) \right].  
\]
By the Taylor expansion, the asymptotic formula
\[
\exp\left(-b \RE \log(1- 2(\cos \Theta_p)p^{-s}+ p^{-2s})\right)
=1+ 2b(\cos \Theta_p) \RE(p^{-s})+O(p^{-2\sigma})
\]
holds for every $p$. 
Then, using \eqref{eq220319}, we obtain
\[
\mathbb{E}\left[ \exp\left(-b \RE \log(1-2(\cos \Theta_p)p^{-s}+ p^{-2s})\right) \right]
=1+O_\epsilon\left( p^{-(\sigma+\frac{1}{2})+\epsilon}+ p^{-2\sigma} \right). 
\]
The series $\sum_{p} p^{-(\sigma+\frac{1}{2})+\epsilon}$ and $\sum_{p} p^{-2\sigma}$ converge for $\sigma>1/2$ with $\epsilon>0$ small enough. 
Hence we conclude that $\mathbb{E}[\exp(b \RE \log{L}(s,\Theta))]$ is finite by \eqref{eq251719}. 
Then it yields
\begin{align*}
&\mathbb{E}[ \exp(a|\RE \log{L}(s,\Theta)|) ]\\
&\leq \mathbb{E}[ \exp(a \RE \log{L}(s,\Theta)) ]
+\mathbb{E}[ \exp(-a \RE \log{L}(s,\Theta)) ]
<\infty 
\end{align*}
as desired. 
The result for $\mathbb{E}[ \exp( a |\IM \log{L}(s,\Theta)|) ]$ is proved similarly. 
\end{proof}

\begin{lemma}\label{lemIP}
Let $s=\sigma+it$ be a fixed complex number with $\sigma>1/2$. 
Suppose that $\Theta=(\Theta_p)$ satisfies condition \eqref{eq220319}. 
Then the function $\widetilde{\mathcal{M}}_s(z,z',\Theta)$ is defined for any $z,z' \in \mathbb{C}$. 
Furthermore, it satisfies
\begin{gather}\label{eq252200}
\widetilde{\mathcal{M}}_s(z,z',\Theta)
=\prod_{p} \widetilde{\mathcal{M}}_{s,p}(z,z',\Theta) 
\end{gather}
for any $z,z' \in \mathbb{C}$, where $\widetilde{\mathcal{M}}_{s,p}(z,z',\Theta)$ is defined as 
\[
\widetilde{\mathcal{M}}_{s,p}(z,z',\Theta)
=\mathbb{E}\left[ \psi_{z,z'}\left(-\log(1-2(\cos\Theta_p)p^{-s}+ p^{-2s})\right) \right]. 
\]
\end{lemma}

\begin{proof}
Note that the quasi-character $\psi_{z,z'}$ satisfies
\begin{gather}\label{eq:06091335}
|\psi_{z,z'}(w)|
\leq \exp(a |\RE w|) \exp(a |\IM w|)
\end{gather}
for any $z,z',w \in \mathbb{C}$, where we put $a=\max\{|z|, |z'|\}$. 
Denote again by $L_Y(s, \Theta)$ the random variable defined as in \eqref{eq110122}. 
Since $L_Y(s,\Theta) \to L(s,\Theta)$ as $Y \to\infty$ almost surely, we obtain $|L_Y(s,\Theta)-L(s,\Theta)|\leq1$ almost surely when $Y$ is large. 
Therefore, we deduce from \eqref{eq:06091335} that the inequality
\begin{align*}
|\psi_{z,z'}(\log{L}_Y(s,\Theta))|
&\leq \exp(a |\RE \log{L}_Y(s,\Theta)|) \exp(a |\IM \log{L}_Y(s,\Theta)|) \\
&\leq \exp(a |\RE \log{L}(s,\Theta)|) \exp(a |\IM \log{L}(s,\Theta)|) \exp(2a)
\end{align*}
holds with $a=\max\{|z|, |z'|\} \geq0$ almost surely. 
Furthermore, the expected value 
\[
\mathbb{E}\left[\exp(a |\RE \log{L}(s,\Theta)|) \exp(a |\IM \log{L}(s,\Theta)|)\right]
\]
is finite by Lemma \ref{lemMGF} and the Cauchy--Schwarz inequality. 
Therefore, the dominated convergence theorem yields that $\widetilde{\mathcal{M}}_s(z,z',\Theta)$ exists, and we have 
\[
\widetilde{\mathcal{M}}_s(z,z',\Theta)
=\lim_{Y \to\infty} \mathbb{E}\left[\psi_{z,z'}(\log{L}_Y(s,\Theta))\right]. 
\]
Since $\Theta=(\Theta_p)$ is independent, we derive
\begin{align*}
\lim_{Y \to\infty} \mathbb{E}\left[\psi_{z,z'}(\log{L}_Y(s,\Theta))\right] 
&=\lim_{Y \to\infty} 
\prod_{p \leq Y} \mathbb{E}\left[ \psi_{z,z'}\left(-\log(1-2(\cos\Theta_p)p^{-s}+p^{-2s})\right) \right] \\
&=\prod_p\widetilde{\mathcal{M}}_{s,p}(z,z',\Theta).  
\end{align*}
Hence the result follows. 
\end{proof}

\begin{proof}[Proof of Proposition \ref{propUB}]
For any $z,w \in \mathbb{C}$, we express $\psi_{z,\overline{z}}(w)$ as
\[
\psi_{z,\overline{z}}(w)
=\exp( i \langle z,w \rangle), 
\]
where $\langle z,w \rangle=\RE{z}\RE{w}+\IM{z}\IM{w} \in \mathbb{R}$.
Then the inequality $|\widetilde{\mathcal{M}}_{s,p}(z,\overline{z},\Theta)| \leq1$ holds for every $p$ since $|\psi_{z,\overline{z}}(w)|=1$. 
Hence formula \eqref{eq252200} with $z'=\overline{z}$ implies
\begin{gather}\label{eq252201}
\left| \widetilde{\mathcal{M}}_s(z,\overline{z},\Theta) \right|
\leq \prod_{p>Q} \left| \widetilde{\mathcal{M}}_{s,p}(z,\overline{z},\Theta) \right|, 
\end{gather}
where $Q>0$ is a real number chosen later. 
Let $s=\sigma+it$ be a fixed complex number with $\sigma>1/2$ and $t \neq0$. 
We estimate $\widetilde{\mathcal{M}}_{s,p}(z,\overline{z},\Theta)$ for $p>Q$ as follows. 
First, we recall the asymptotic formula
\[
\psi_{z,\overline{z}}(w)
=1+ i \langle z,w \rangle- \frac{1}{2} \langle z,w \rangle^2+ O(|z|^3|w|^3)
\]
which is valid for $|z||w|<c$ with any positive constant $c$. 
We have also
\[
-\log( 1-2(\cos \Theta_p)p^{-s}+p^{-2s} )
=2(\cos \Theta_p)p^{-s}+ (\cos 2\Theta_p)p^{-2s}+ O(p^{-3\sigma}). 
\]
Combining them, we deduce
\begin{align}\label{eq261706}
\widetilde{\mathcal{M}}_{s,p}(z,\overline{z},\Theta)
&=1
+2i \mathbb{E}[\cos \Theta_p] \langle z, p^{-s} \rangle
+i \mathbb{E}[\cos 2\Theta_p] \langle z, p^{-2s} \rangle \\
&\qquad
-2\mathbb{E}[(\cos \Theta_p)^2] \langle z, p^{-s} \rangle^2
+O\left( |z|^3p^{-3\sigma} \right) \nonumber
\end{align}
if $|z|p^{-\sigma}<c$ is satisfied. 
Hence there exists a small constant $c_1>0$ such that, for any $0<c<c_1$, we have the inequality $|\widetilde{\mathcal{M}}_{s,p}(z,\overline{z},\Theta)-1|<1/2$ if $p$ satisfies
\begin{gather}\label{eq262659}
p
>Q(c, z,\sigma)
:=\left( \frac{|z|}{c} \right)^{1/\sigma}. 
\end{gather}
Let $\Log$ denote the principal branch of logarithm. 
For $p>Q(c, z,\sigma)$, we obtain
\begin{align*}
\Log \widetilde{\mathcal{M}}_{s,p}(z,\overline{z},\Theta)
&= 2i \mathbb{E}[\cos \Theta_p] \langle z, p^{-s} \rangle
+i \mathbb{E}[\cos 2\Theta_p] \langle z, p^{-2s} \rangle \\
&\qquad
-2\mathbb{E}[(\cos \Theta_p)^2] \langle z, p^{-s} \rangle^2
+O_\epsilon\left( |z|^2 p^{-2\sigma-1+2\epsilon}+ |z|^3 p^{-3\sigma} \right)
\end{align*}
by applying the formula $\Log(1+w)=w+O(|w|^2)$ along with \eqref{eq220319}. 
Note that we have $2\langle z, p^{-s} \rangle^2= |z|^2 p^{-2\sigma} +\langle z^2, p^{-2s} \rangle$. 
Then, taking the real parts, we deduce
\begin{align*}
\log \left|\widetilde{\mathcal{M}}_{s,p}(z,\overline{z},\Theta)\right|
&= -\mathbb{E}[(\cos \Theta_p)^2] |z|^2 p^{-2\sigma}
-\mathbb{E}[(\cos \Theta_p)^2] \langle z^2, p^{-2s} \rangle\\
&\qquad
+O\left( |z|^2 p^{-2\sigma-\frac{1}{2}}+ |z|^3 p^{-3\sigma} \right), 
\end{align*}
where $\epsilon$ is taken with $0<\epsilon<1/4$. 
By \eqref{eq220320}, the inequality 
\begin{gather}\label{eq261645}
\log \left|\widetilde{\mathcal{M}}_{s,p}(z,\overline{z};\Theta)\right|
\leq -\delta |z|^2 p^{-2\sigma}
-\delta \langle z^2, p^{-2s} \rangle
+K_1 (|z|^2 p^{-2\sigma-\frac{1}{2}}+ |z|^3 p^{-3\sigma})
\end{gather}
holds with positive absolute constants $\delta$ and $K_1$. 
We have
\begin{align*}
\sum_{p>Q} |z|^2 p^{-2\sigma}
&=|z|^2 \frac{1}{2\sigma-1} \frac{Q^{1-2\sigma}}{\log{Q}}
+O_\sigma\left( |z|^2 \frac{Q^{1-2\sigma}}{(\log{Q})^2} \right),\\
\sum_{p>Q} \langle z^2, p^{-2s} \rangle
&=\left\langle z^2, \frac{1}{2s-1} \frac{Q^{1-2s}}{\log{Q}} \right\rangle
+O_s\left( |z|^2 \frac{Q^{1-2\sigma}}{(\log{Q})^2} \right)
\end{align*}
for any $Q>3$ by the prime number theorem.
Since the inequality 
\[
\left| \left\langle z^2, \frac{1}{2s-1} \frac{Q^{1-2s}}{\log{Q}} \right\rangle \right|
\leq |z|^2 \frac{1}{|2s-1|} \frac{Q^{1-2\sigma}}{\log{Q}}
\]
holds, we obtain 
\[
-\sum_{p>Q} |z|^2 p^{-2\sigma}
-\sum_{p>Q} \langle z^2, p^{-2s} \rangle
\leq -d(s) |z|^2 \frac{1}{2\sigma-1} \frac{Q^{1-2\sigma}}{\log{Q}}
+K_2(s) |z|^2 \frac{Q^{1-2\sigma}}{(\log{Q})^2}, 
\]
where $d(s)$ is a constant given by $d(s)=1- (2\sigma-1)|2s-1|^{-1}$, and $K_2(s)$ is a positive constant.
By the assumption that $s=\sigma+it$ with $t \neq0$, we know $d(s)>0$. 
As a result, we deduce from \eqref{eq261645} that 
\begin{align*}
\sum_{p>Q} \log \left|\widetilde{\mathcal{M}}_{s,p}(z,\overline{z},\Theta)\right|
&\leq -\frac{d(s) \delta}{2\sigma-1} \frac{|z|^2 Q^{1-2\sigma}}{\log{Q}}
+\frac{K_2(s)}{\log{Q}} \frac{|z|^2 Q^{1-2\sigma}}{\log{Q}}\\
&\qquad
+\frac{K_3(s)}{Q^{1/2}} \frac{|z|^2 Q^{1-2\sigma}}{\log{Q}}
+c K_3(s) \frac{|z|^2 Q^{1-2\sigma}}{\log{Q}}
\end{align*}
for $Q=Q(c,z,\sigma)$ given in \eqref{eq262659}, where $K_3(s)$ is a positive constant.
If we choose $c>0$ as a suitably small constant depending on $s$, then it holds 
\[
\sum_{p>Q} \log \left|\widetilde{\mathcal{M}}_{s,p}(z,\overline{z},\Theta)\right|
\leq -c(s) \frac{|z|^{1/\sigma}}{\log|z|}
\]
with some constant $c(s)>0$. 
From \eqref{eq252201} we obtain the conclusion in the case $t \neq0$. 
The proof for $s=\sigma>1/2$ is presented in a similar way. 
The difference is coming from the fact that $\widetilde{\mathcal{M}}_{\sigma,p}(z,\overline{z};\Theta)$ is represented as 
\[
\widetilde{\mathcal{M}}_{\sigma,p}(z,\overline{z},\Theta)
=\mathbb{E}\left[ \exp\left(-x \log(1- 2(\cos\Theta_p)p^{-\sigma}+ p^{-2\sigma})\right) \right]
\]
for $z=x+iy$. 
Hence, in place of \eqref{eq261706}, the asymptotic formula
\begin{align*}
\widetilde{\mathcal{M}}_{\sigma,p}(z,\overline{z},\Theta)
&=1
+2i \mathbb{E}[\cos \Theta_p] x p^{-\sigma}
+i \mathbb{E}[\cos 2\Theta_p] x p^{-2\sigma} \\
&\qquad
-2 \mathbb{E}[(\cos \Theta_p)^2] x^2 p^{-2\sigma}
+O\left(|x|^3 p^{-3\sigma}\right)
\end{align*} 
is available when $|x| p^{-\sigma}<c$ is satisfied. 
Then it yields
\[
\log \left|\widetilde{\mathcal{M}}_{\sigma,p}(z,\overline{z},\Theta)\right|
\leq -2 \delta x^2 p^{-2\sigma}
+K( x^2 p^{-2\sigma-\frac{1}{2}}+ |x|^3 p^{-3\sigma}) 
\]
with some $K>0$. 
Therefore we obtain the result by estimating $\sum_{p>Q} x^2 p^{-2\sigma}$. 
\end{proof}

\subsection{Strategy for the proof of Proposition \ref{propCM1}}\label{sec2.2}
Let $q$ be a large prime number. 
For $Y<q$, we define the Dirichlet polynomials
\[
R_Y(s,f)
=\mathop{ \sum_{p} \sum_{m=1}^\infty} \limits_{p^m \leq Y}
\frac{b(p^m,f)}{p^{ms}}
\quad\text{and}\quad
R_Y(s,\Theta)
=\mathop{ \sum_{p} \sum_{m=1}^\infty} \limits_{p^m \leq Y}
\frac{b(p^m,\Theta)}{p^{ms}} 
\]
for $\RE(s)>1/2$, where the coefficient $b(p^m,f)$ is given by \eqref{eq052248}, and we put
\[
b(p^m,\Theta)
=\frac{2\cos(m\Theta_p)}{m}. 
\]
As analogues of the functions $\widetilde{\mathcal{M}}_{s,q}(z,z',\omega_q)^\mathcal{E}$ and $\widetilde{\mathcal{M}}_s(z,z',\Theta)$, we further define
\begin{align*}
\widetilde{\mathcal{M}}_{s,q}(z,z',\omega_q;Y)^\mathcal{E}
&=\sum_{f \in B_2(q) \setminus \mathcal{E}} \omega_q(f) \psi_{z,z'}(R_Y(s,f)), \\
\widetilde{\mathcal{M}}_s(z,z',\Theta;Y)
&=\mathbb{E}\left[ \psi_{z,z'}(R_Y(s,\Theta)) \right]. 
\end{align*}
The first step to the proof of Proposition \ref{propCM1} is to approximate $\log{L}(s,f)$ and $\log{L}(s,\Theta)$ by $R_Y(s,f)$ and $R_Y(s,\Theta)$, respectively. 
Then, we prove the following proposition which connects $\widetilde{\mathcal{M}}_{s,q}(z,z',\omega_q;Y)^\mathcal{E}$ and $\widetilde{\mathcal{M}}_s(z,z',\Theta;Y)$. 

\begin{proposition}\label{propCM2}
Under the assumptions of Proposition \ref{propCM1}, we put $Y=(\log{q})^B$ and define a set $\mathcal{E}_1=\mathcal{E}_1(q,s;Y)$ as $\mathcal{E}_1(q,s;Y)=\emptyset$ if $\sigma \geq1$; and
\begin{gather}\label{eq271607}
\mathcal{E}_1(q,s;Y)
=\left\{ f \in B_2(q) ~\middle|~ |R_Y(s,f)|>(\log{q})^{1-\sigma}(\log\log{q})^{-1} \right\}
\end{gather}
if $1/2<\sigma<1$. 
Then there exist positive constants $a_2=a_2(\sigma,B)$ and $b_2=b_2(\sigma,B)$ such that 
\[
\widetilde{\mathcal{M}}_{s,q}(z,z',\omega_q;Y)^{\mathcal{E}_1}
=\widetilde{\mathcal{M}}_s(z,z',\Theta;Y)
+O\left( \exp\left( -b_2 \frac{\log{q}}{\log\log{q}} \right) \right)
\]
for all $z, z'\in\mathbb{C}$ with $\max\{|z|,|z'|\} \leq a_2 R_\sigma(q)$, where $R_\sigma(q)$ is defined as in \eqref{eq301946}, and the implied constant depends only on $s$ and $B$. 
\end{proposition}

A similar result was proved in \cite[Proposition 6.1]{Lamzouri2011b}, but we use a different finite truncation in this paper. 
Proposition \ref{propCM2} is also an analogue of \cite[Proposition 2.3]{LamzouriLesterRadziwill2019} for automorphic $L$-functions. 
The proof of Proposition \ref{propCM2} is given in Section \ref{sec3.3}. 
In the remaining part of this subsection, we prove some preliminary lemmas. 

\begin{lemma}\label{lemLS}
Let $s=\sigma+it$ be a fixed complex number with $\sigma>1/2$. 
Let $1\ll y \leq z$ be large real numbers. 
Then we have 
\[
\sum_{f \in B_2(q)} \frac{1}{\langle f,f \rangle}
\Bigg| \sum_{y \leq p \leq z} \frac{a(p,f)}{p^s} \Bigg|^{2k}
\ll 2^{2k} \frac{(2k)!}{k!} \Bigg( \sum_{y \leq p \leq z} \frac{1}{p^{2\sigma}} \Bigg)^k
+\frac{\log{q}}{\sqrt{q}}
\]
for $k \in \mathbb{Z}$ with $1 \leq k \leq \log{q}/(2\log{z})$. 
The implied constant is absolute. 
\end{lemma}

\begin{proof}
Lamzouri \cite[Lemma 6.5]{Lamzouri2011b} proved the same result when $s$ is a fixed real number. 
We can show Lemma \ref{lemLS} in a similar way by noting that $a(p,f)$ is always real. 
\end{proof}

\begin{lemma}\label{lemES1}
Let $s=\sigma+it$ be a fixed complex number with $1/2<\sigma<1$. 
For $Y=(\log{q})^B$ with $B\geq1$, we define the subset $\mathcal{E}_1(q,s;Y)$ as in \eqref{eq271607}. 
Then there exists a positive constant $b_3=b_3(\sigma,B)$ such that 
\[
\frac{\# \mathcal{E}_1(q,s;Y)}{q}
\ll \exp\left( -b_3 \frac{\log{q}}{\log\log{q}} \right), 
\]
where the implied constant depends only on $s$ and $B$.
\end{lemma}

\begin{proof}
For a large integer $k \leq \log{q}/(2B\log\log{q})$, we obtain 
\begin{gather}\label{eq310121}
R_Y(s,f)
=\sum_{p< k \log{k}} \frac{a(p,f)}{p^s}
+\sum_{k \log{k} \leq p \leq Y} \frac{a(p,f)}{p^s}
+O(\log\zeta(2\sigma)). 
\end{gather}
First, we estimate the $2k$-th moment
\[
S_1
=\sum_{f \in B_2(q)} \frac{1}{\langle f,f \rangle}
\Bigg| \sum_{p< k \log{k}} \frac{a(p,f)}{p^s} \Bigg|^{2k}. 
\]
The Petersson trace formula yields that $\sum_{f \in B_2(q)} (4 \pi \langle f,f \rangle)^{-1}$ is bounded, and we have $|a(p,f)|\leq2$ for all $p$. 
Hence we evaluate $S_1$ as
\begin{gather}\label{eq310122}
S_1
\ll \Bigg( \sum_{p< k \log{k}} \frac{2}{p^\sigma} \Bigg)^{2k} 
\ll \left( \frac{4 k^{1-\sigma}}{(1-\sigma)(\log{k})^\sigma} \right)^{2k}
\end{gather}
by applying the prime number theorem. 
Next, we consider
\[
S_2
=\sum_{f \in B_2(q)} \frac{1}{\langle f,f \rangle}
\left| \sum_{k \log{k} \leq p \leq Y} \frac{a(p,f)}{p^s} \right|^{2k}.  
\]
Using Lemma \ref{lemLS}, it is estimated as
\begin{gather}\label{eq310123}
S_2
\ll 2^{2k} \frac{(2k)!}{k!} \Bigg( \sum_{k \log k \leq p \leq Y} \frac{1}{p^{2\sigma}} \Bigg)^k
+\frac{\log{q}}{\sqrt{q}}
\ll_\sigma \left( \frac{4k^{1-\sigma}}{\sqrt{2\sigma-1} (\log{k})^\sigma} \right)^{2k}. 
\end{gather}
Then, asymptotic formula \eqref{eq310121} yields the inequality
\[
\sum_{f \in B_2(q)} \frac{1}{\langle f,f \rangle}
|R_Y(s,f)|^{2k}
\leq 9^k S_1+ 9^k S_2+ 9^k (C \log{\zeta}(2\sigma))^{2k}, 
\]
where $C>0$ is an absolute constant. 
Thus it is deduced from \eqref{eq310122} and \eqref{eq310123} that 
\[
\sum_{f \in B_2(q)} \frac{1}{\langle f,f \rangle}
|R_Y(s,f)|^{2k}
\ll \left( K(\sigma) \frac{k^{1-\sigma}}{(\log{k})^\sigma} \right)^{2k}
\]
with some constant $K(\sigma)>0$. 
Hence we obtain
\begin{align*}
\sum_{f \in \mathcal{E}_1(q,s;Y)} \frac{1}{\langle f,f \rangle}
&\leq \left( \frac{(\log{q})^{1-\sigma}}{\log\log{q}} \right)^{-2k}
\sum_{f \in B_2(q)} \frac{1}{\langle f,f \rangle} |R_Y(s,f)|^{2k}\\
&\ll \left( K(\sigma) \frac{k^{1-\sigma}\log\log{q}}{(\log{k})^\sigma (\log{q})^{1-\sigma}} \right)^{2k}. 
\end{align*}
Let $k=\lfloor \log{q}/(L\log\log{q}) \rfloor$ with a constant $L=L(\sigma,B)>0$ large enough. 
Then there exists a positive constant $b(\sigma,B)$ such that
\[
\sum_{f \in \mathcal{E}_1(q,s;Y)} \frac{1}{\langle f,f \rangle}
\ll \exp\left( -b(\sigma,B) \frac{\log{q}}{\log\log{q}} \right)  
\]
with the implied constant depending only on $s$ and $B$. 
Since we have the lower bound $\langle f,f \rangle^{-1} \gg q^{-1}(\log{q})^{-3}$, the desired result follows. 
\end{proof}

\begin{lemma}\label{lemES2}
Let $s=\sigma+it$ be a fixed complex number with $\sigma>1/2$. 
For $Y=(\log{q})^B$ with $B \geq1$, we define 
\begin{gather}\label{eq282331}
\mathcal{E}_2(q,s;Y)
=\left\{ f \in B_2(q) ~\middle|~ |\log{L}(s,f)-R_Y(s,f)| > Y^{-\frac{1}{2}(\sigma-\frac{1}{2})} (\log{q})^2 \right\}. 
\end{gather}
Then there exists a positive constant $b_4=b_4(\sigma,B)$ such that
\[
\frac{\# \mathcal{E}_2(q,s;Y)}{q}
\ll \exp\left( -b_4 \frac{\log{q}}{\log\log{q}} \right)
\]
with the implied constant depending only on $s$ and $B$. 
\end{lemma}

\begin{proof}
By the asymptotic formula \cite[Lemma 4.4]{CogdellMichel2004}, we have 
\[
\log{L}(s,f)
=R_Y(s,f)
+O_\sigma\left( Y^{-\frac{1}{2}(\sigma-\frac{1}{2})} (\log{q}) \right)
\]
for $2q \geq |t|$ if there exist no zeros of $L(s,f)$ inside the rectangle 
\[
A_s(Y)
=\left\{ z \in \mathbb{C} ~\middle|~ \sigma_0 \leq \RE(z) \leq1,~ |\IM(z)-t| \leq Y+3 \right\},
\] 
where $\sigma_0=\frac{1}{2}(\sigma+\frac{1}{2})$. 
In other words, the condition $f \in \mathcal{E}_2(q,s;Y)$ implies the existence of zeros of $L(s,f)$ inside the rectangle $A_s(Y)$. 
Thus we have 
\[
\# \mathcal{E}_2(q,s;Y)
\leq \sum_{f \in B_2(q)} N(f; \sigma_0, t-Y-3, t+Y+3), 
\]
where $N(f; \alpha, t_1, t_2)$ counts the number of zeros $\rho=\beta+i \gamma$ of $L(s,f)$ such that $\beta \geq \alpha$ and $t_1 \leq \gamma \leq t_2$. 
Furthermore, we apply the zero density estimate of Kowalski--Michel \cite[Theorem 4]{KowalskiMichel1999} to derive
\[
\sum_{f \in B_2(q)} N(f; \sigma_0, t-Y-3, t+Y+3)
\ll q^{1-\frac{1}{10}(\sigma-\frac{1}{2})} (\log{q})^{K(B)}
\]
for $q \geq e^{|t|}$ with a constant $K(B)>0$. 
Hence we obtain the conclusion. 
\end{proof}

\begin{lemma}\label{lemM}
Let $s=\sigma+it$ be a fixed complex number with $\sigma>1/2$ and $Y>0$ be a large real number. 
Suppose that $\Theta=(\Theta_p)$ satisfies condition \eqref{eq220319}. 
Then we have
\[
\widetilde{\mathcal{M}}_s(z,z',\Theta)
=\widetilde{\mathcal{M}}_s(z,z',\Theta;Y)
+O\left((1+E_s(z,z',\Theta))|z|Y^{-\frac{1}{2}(\sigma-\frac{1}{2})}\right)
\]
for $z,z'\in\mathbb{C}$ satisfying $z'=\pm{z}$ or $\overline{z}$ with $|z|\leq{Y}^{\frac{1}{2}(\sigma-\frac{1}{2})}$, where $E_s(z,z',\Theta)$ is defined as in \eqref{eq301947}. 
The implied constant depends only on $\sigma$.
\end{lemma}

\begin{proof}
By formula \eqref{eq252200}, we obtain 
\begin{gather}\label{eq021805}
\widetilde{\mathcal{M}}_s(z,z',\Theta)
=\prod_{p \leq Y} \widetilde{\mathcal{M}}_{s,p}(z,z',\Theta)
\cdot
\prod_{p>Y} \widetilde{\mathcal{M}}_{s,p}(z,z',\Theta). 
\end{gather}
The local factor $\widetilde{\mathcal{M}}_{s,p}(z,z',\Theta)$ is represented as 
\[
\widetilde{\mathcal{M}}_{s,p}(z,z',\Theta)
=\mathbb{E}\left[ \psi_{z,z'}\Bigg( \sum_{m \leq \frac{\log{Y}}{\log{p}}} b(p^m,\Theta) p^{-ms}
+B_{p,Y}(s,\Theta) \Bigg) \right] 
\]
for any $p \leq Y$, where 
\[
B_{p,Y}(s,\Theta)
=\sum_{m > \frac{\log{Y}}{\log{p}}} b(p^m,\Theta) p^{-ms}
\ll \min\left\{ Y^{-\sigma}, p^{-2\sigma} \right\}. 
\]
Then, by the independence of $\Theta=(\Theta_p)$, the first product of \eqref{eq021805} is calculated as 
\[
\prod_{p \leq Y} \widetilde{\mathcal{M}}_{s,p}(z,z',\Theta)
=\mathbb{E}\left[ \psi_{z,z'}(R_Y(s,\Theta))
\psi_{z,z'}\Bigg( \sum_{p \leq Y} B_{p,Y}(s,\Theta) \Bigg) \right]. 
\]
Since we have 
\[
\sum_{p \leq Y} B_{p,Y}(s,\Theta)
\ll \sum_{p \leq \sqrt{Y}} Y^{-\sigma}
+\sum_{\sqrt{Y} < p \leq Y} p^{-2\sigma}
\ll_\sigma Y^{-(\sigma-\frac{1}{2})}, 
\]
the asymptotic formula 
\[
\psi_{z,z'}\Bigg( \sum_{p \leq Y} B_{p,Y}(s,\Theta) \Bigg)
=1+O\left( |z| {Y}^{-(\sigma-\frac{1}{2})} \right)
\]
is valid in the range $|z| \leq Y^{\frac{1}{2}(\sigma-\frac{1}{2})}$. 
Therefore, we obtain
\begin{gather}\label{eq062102}
\prod_{p \leq Y} \widetilde{\mathcal{M}}_{s,p}(z,z',\Theta)
=\widetilde{\mathcal{M}}_s(z,z',\Theta;Y)
+O\left( E_s(z,z',\Theta;Y) |z| Y^{-(\sigma-\frac{1}{2})} \right), 
\end{gather}
where $E_s(z,z',\Theta;Y)=\mathbb{E}\left[|\psi_{z,z'}(R_Y(s,\Theta))|\right]$. 
Next, if $p>Y$, the inequality 
\[
|z| \left| \log(1-2(\cos \Theta_p) p^{-s}+ p^{-2s}) \right|
<1
\]
is satisfied by the assumption $|z| \leq Y^{\frac{1}{2}(\sigma-\frac{1}{2})}$. 
Hence we have 
\begin{align*}
\widetilde{\mathcal{M}}_{s,p}(z,z',\Theta)
&=1
+i \mathbb{E}[\cos \Theta_p] (z p^{-\overline{s}}+z' p^{-s})
+O\left(|z|^2 p^{-2\sigma}\right)\\
&=1+O_\sigma\left( |z| p^{-\frac{1}{2}\sigma-\frac{3}{4}}
+|z|^2 p^{-2\sigma} \right)
\end{align*}
for $p>Y$ by condition \eqref{eq220319} with $\epsilon=\frac{1}{2}(\sigma-\frac{1}{2})>0$. 
Then it yields
\begin{align}\label{eq062122}
\prod_{p>Y} \widetilde{\mathcal{M}}_{s,p}(z,z',\Theta)
&=1+ O\Bigg( \sum_{p>Y} |z| p^{-\frac{1}{2}\sigma-\frac{3}{4}}
+\sum_{p>Y} |z|^2 p^{-2\sigma} \Bigg)\\
&=1+ O\left( |z| Y^{-\frac{1}{2}(\sigma-\frac{1}{2})} \right). \nonumber
\end{align}
The work of the estimate for $E_s(z,z',\Theta;Y)$ is remaining. 
Remark that the equality $|\psi_{z,z'}(w)|=\psi_{\xi,\xi'}(w)$ holds with 
\begin{gather}\label{eq092258}
(\xi,\xi')
=
\begin{cases}
(i \IM(z), i \IM(z)) 
& \text{if $z'=z$}, \\
(\RE(z), -\RE(z)) 
& \text{if $z'=-z$}, \\
(0,0) 
& \text{if $z'=\overline{z}$}. 
\end{cases}
\end{gather}
Furthermore we have $|\psi_{\xi,\xi'}(w)|=\psi_{\xi,\xi'}(w)$. 
Hence we apply \eqref{eq021805}, \eqref{eq062102}, and \eqref{eq062122} to deduce the bound $E_s(z,z',\Theta;Y) \ll E_s(z,z',\Theta)$. 
Thus, formula \eqref{eq062102} yields
\[
\prod_{p \leq Y} \widetilde{\mathcal{M}}_{s,p}(z,z',\Theta)
=\widetilde{\mathcal{M}}_s(z,z',\Theta;Y)
+O\left( E_s(z,z',\Theta) |z| Y^{-(\sigma-\frac{1}{2})} \right),  
\]
which completes the proof by \eqref{eq021805} and \eqref{eq062122}.
\end{proof}

\begin{lemma}\label{lemIM}
Let $s=\sigma+it$ be a fixed complex number with $\sigma>1/2$. 
Suppose that $\Theta=(\Theta_p)$ satisfies condition \eqref{eq220323}. 
For $Y=(\log{q})^B$ with $B \geq1$, there exists a positive constant $K_1=K_1(\sigma,B)$ such that 
\[
\mathbb{E}\left[|R_Y(s,\Theta)|^{2k}\right]
\ll
\begin{cases}
K_1^{2k} 
& \text{if $\sigma>1$}, \\
(K_1\log\log\log{q})^{2k} 
& \text{if $\sigma=1$}, \\
\displaystyle{ \left( K_1 \frac{k^{1-\sigma}}{(\log2k)^\sigma} \right)^{2k} } 
& \text{if $1/2<\sigma<1$} 
\end{cases}
\]
for any integer $k \geq1$. 
\end{lemma}

\begin{proof}
Let $\sigma\geq1$. 
In this case we have the estimate
\[
|R_Y(s,\Theta)|
\leq \sum_{p \leq Y} \sum_{m=1}^{\infty} \frac{2}{m} p^{-m\sigma}
\ll
\begin{cases}
\log\zeta(\sigma) 
& \text{for $\sigma>1$}, \\
\log\log{Y} 
& \text{for $\sigma=1$}
\end{cases}
\]
by the inequality $|b(p^m,\Theta)| \leq 2/m$. 
Hence the desired estimate follows for $\sigma\geq1$. 
Let $1/2<\sigma<1$. 
If the inequality $Y \leq Ck \log{2k}$ holds with a constant $C \geq 2$, then we obtain
\[
R_Y(s,\Theta)
=\sum_{p \leq Y} \frac{a(p,\Theta)}{p^s}
+O(\log\zeta(2\sigma))
\ll \frac{C^{1-\sigma} k^{1-\sigma}}{(1-\sigma) (\log2k)^\sigma} 
\]
by the prime number theorem, which yields the result. 
Therefore we suppose the inequality $C k \log{2k} < Y$ below. 
In that case, we have 
\begin{gather}\label{eq211628}
R_Y(s,\Theta)
=\sum_{p < Ck \log{2k}} \frac{a(p,\Theta)}{p^s}
+\sum_{C k \log{2k} \leq p \leq Y} \frac{a(p,\Theta)}{p^s}
+O(\log\zeta(2\sigma)). 
\end{gather}
The contribution of the first sum is estimated as
\begin{gather}\label{eq211629}
\mathbb{E}\Bigg[\Bigg| \sum_{p<C k \log{2k}} \frac{a(p,\Theta)}{p^s} \Bigg|^{2k}\Bigg]
\ll \left( \frac{4 C^{1-\sigma} k^{1-\sigma}}{(1-\sigma) (\log2{k})^\sigma} \right)^{2k}.
\end{gather}
Let $1\ll y<z$ be large real numbers. 
Next, we obtain
\begin{align}\label{eq211714}
&\mathbb{E}\Bigg[ \Bigg 
|\sum_{y \leq p \leq z} \frac{a(p,\Theta)}{p^s} \Bigg|^{2k} \Bigg] \nonumber\\
&\ll \sum_{y \leq p_1\leq z} \cdots \sum_{y \leq p_{2k} \leq z}
\frac{1}{(p_1 \cdots p_{2k})^\sigma}
\left| \mathbb{E}[a(p_1,\Theta) \cdots a(p_{2k},\Theta)] \right| \nonumber\\
&=\sum_{n=1}^{2k}
\mathop{ \sum_{y \leq p_1 \leq z } \cdots \sum_{y \leq p_{2k} \leq z} }
\limits_{ \omega(p_1 \cdots p_{2k})=n }
\frac{1}{(p_1 \cdots p_{2k})^\sigma}
\left| \mathbb{E}[a(p_1,\Theta) \cdots a(p_{2k},\Theta)] \right| \nonumber\\
&=\sum_{n=1}^{2k}
\mathop{ \sum_{y \leq p_1 \leq z } \cdots \sum_{y \leq p_n \leq z} }
\limits_{ \text{$p_j$ are distinct} }\\
&\qquad\quad
\sum_{\substack{m_1+\cdots+m_n=2k \\ \forall j,~ m_j \geq1}}
\binom{2k}{m_1,\ldots,m_n}
\frac{\left| \mathbb{E}[a(p_1,\Theta)^{m_1}] \right| \cdots \left| \mathbb{E}[a(p_n,\Theta)^{m_n}] \right|}
{(p_1^{m_1} \cdots p_j^{m_j})^\sigma}, \nonumber
\end{align}
where $\omega(n)$ stands for the number of distinct prime factors of $n$. 
Recall that the Chebyshev polynomials $U_m(x)$ is an odd polynomial of degree $m$ when $m$ is odd. 
Furthermore, we have $\mathbb{E}[U_m(\cos \Theta_p)]=0$ for odd $m$ by \eqref{eq220323}. 
By induction, one can show that $\mathbb{E}[a(p,\Theta)^{m}]$ vanishes unless $m$ is even. 
Hence \eqref{eq211714} equals to
\begin{align*}
&=\sum_{n=1}^{k}
\mathop{ \sum_{y \leq p_1 \leq z} \cdots \sum_{y \leq p_n \leq z} }
\limits_{ \text{$p_j$ are distinct} }\\
&\qquad\quad
\sum_{\substack{m_1+\cdots+m_n=k \\ \forall j,~ m_j \geq1}}
\binom{2k}{2m_1,\ldots,2m_n}
\frac{\left| \mathbb{E}[a(p_1,\Theta)^{2m_1}] \right| \cdots \left| \mathbb{E}[a(p_n,\Theta)^{2m_n}] \right|}
{(p_1^{m_1} \cdots p_j^{m_j})^{2\sigma}},  
\end{align*}
which is evaluated as
\begin{align*}
&\ll 2^{2k} \frac{(2k)!}{k!} \sum_{n=1}^{k}
\mathop{ \sum_{y \leq p_1 \leq z} \cdots \sum_{y \leq p_n \leq z} }
\limits_{ \text{$p_j$ are distinct} }
\sum_{\substack{m_1+\cdots+m_n=k \\ \forall j,~ m_j \geq1}}
\binom{k}{m_1,\ldots,m_n}
\frac{1}{(p_1^{m_1} \cdots p_j^{m_j})^{2\sigma}}\\
&=2^{2k} \frac{(2k)!}{k!}
\Bigg( \sum_{y \leq p \leq z} \frac{1}{p^{2\sigma}} \Bigg)^{k}
\end{align*}
by using the inequalities $|a(p,\Theta)| \leq2$ and 
\[
\binom{2k}{2m_1,\ldots,2m_n}
\leq \frac{(2k)!}{k!} \binom{k}{m_1,\ldots,m_n}.
\]
Finally, taking $y= C k \log{2k}$ and $z=Y$ with $C \geq2$ sufficiently large, we arrive at
\begin{gather}\label{eq211630}
\mathbb{E}\Bigg[ \Bigg| \sum_{C k \log{2k} \leq p \leq Y}
\frac{a(p,\Theta)}{p^s} \Bigg|^{2k} \Bigg]
\ll \left( \frac{4 C^{\frac{1}{2}-\sigma} k^{1-\sigma}}{\sqrt{2\sigma-1} (\log{2k})^\sigma} \right)^{2k}.
\end{gather}
By \eqref{eq211628}, \eqref{eq211629}, and \eqref{eq211630} along with $(a+b+c)^{2k} \leq 9^k (a^{2k}+ b^{2k}+ c^{2k})$, we deduce the conclusion. 
\end{proof}

\section{Complex moments of automorphic $L$-functions}\label{sec3}

\subsection{Results on Dirichlet coefficients and applications}\label{sec3.1}
In Section \ref{sec1.1}, we introduced two Dirichlet coefficients $a(p^m,f)$ and $b(p^m,f)$; see \eqref{eq070007} and \eqref{eq052248}. 
We also defined $a(p^m,\Theta)$ and $b(p^m,\Theta)$ by regarding the real number $\theta_p(f)$ as a random variable $\Theta_p$. 
First, we show that they are connected by the following relations. 

\begin{lemma}\label{lemAB}
Let $f \in B_2(q)$ be a cusp form and $\Theta=(\Theta_p)$ be a sequence of independent $[0,\pi]$-valued random variable. 
Then we have
\[
b(p^m,f)
=\sum_{j=0}^{m} c_m(j) a(p^j,f)
\quad\text{and}\quad
b(p^m,\Theta)
=\sum_{j=0}^{m} c_m(j) a(p^j,\Theta)
\]
for all prime numbers $p \neq q$ and $m \geq1$, where 
\[
c_m(j)
=
\begin{cases}
1/m 
& \text{if $j=m$}, \\
-1/m 
& \text{if $j=m-2$}, \\
0 
& \text{otherwise.}
\end{cases}
\]
\end{lemma}

\begin{proof}
Recall that $\{U_j(\cos \theta)\}_{j \geq0}$ is an orthonormal basis of $L^2([0,\pi])$ with respect to the Sato--Tate measure $\mu_\infty$. 
Hence we obtain
\begin{gather}\label{eq071651}
\frac{2 \cos(m \theta)}{m}
=\sum_{j=0}^{\infty} c_m(j) U_j(\cos\theta)
\end{gather}
for any $\theta \in [0,\pi]$, where the coefficient $c_m(j)$ is determined by
\[
c_m(j)
=\int_{0}^{\pi} \frac{2 \cos(m \phi)}{m} U_j(\cos \phi) \,d \mu_\infty(\phi)
=\frac{4}{m \pi} \int_{0}^{\pi} \cos(m \phi) \sin((j+1) \phi) \sin \phi \,d\phi.
\]
The integral vanishes except for $j=m$ or $m-2$. 
We have also $c_m(m)=1/m$ and $c_m(m-2)=-1/m$. 
Hence, putting $\theta=\theta_p(f)$ in formula \eqref{eq071651}, we obtain the former statement by the definitions of $a(p^m,f)$ and $b(p^m,f)$. 
Similarly, the latter one is proved by letting $\theta=\Theta_p$ in \eqref{eq071651}. 
\end{proof}

Rudnick--Soundararajan \cite{RudnickSoundararajan2006} introduced a ring $\mathcal{H}$ generalized over the integers by symbols $x(1), x(2), \ldots$ with the Hecke relations
\[
x(1)
=1
\quad\text{and}\quad
x(m)x(n)
=\sum_{d \mid (m,n)} x \left( \frac{mn}{d^2} \right). 
\]
For any $n \geq1$, we regard $a(n,f)$ as a member of $\mathcal{H}$ and $a(n,\Theta)$ as an $\mathcal{H}$-valued random variable since they satisfy the above Hecke relations. 
Several properties on the ring $\mathcal{H}$ are seen in \cite{Lamzouri2011b, RudnickSoundararajan2006}. 
In particular, it holds that 
\begin{gather}\label{eq072149}
x(n_1) \cdots x(n_r)
=\sum_{n \mid \prod_{k=1}^{r} n_k} b_n(n_1,\ldots,n_r) x(n)
\end{gather}
with a non-negative integer $b_n(n_1,\ldots,n_r)$. 
We have the upper bound
\begin{gather}\label{eq072150}
b_n(n_1,\ldots,n_r)
\leq d(n_1) \cdots d(n_r)
\end{gather}
for all $n, n_1,\cdots, n_r \geq1$, where $d(n)$ indicates the number of positive divisors of $n$. 

\begin{lemma}\label{lemAveB}
Let $q$ be a large prime number. 
Let $p_1,\ldots,p_r$ be prime numbers with $p_j \neq q$ for all $j$, and let $m_1,\ldots,m_r \geq1$. 
Suppose that $\Theta=(\Theta_p)$ and $\omega_q$ satisfies condition \eqref{eq220321}. 
For each $\epsilon>0$, we obtain
\begin{align*}
&\sum_{f \in B_2(q)} \omega_q(f) b(p_1^{m_1},f) \cdots b(p_r^{m_r},f) \\
&=\mathbb{E}[ b(p_1^{m_1},\Theta) \cdots b(p_r^{m_r},\Theta) ]
+O\left( (p_1^{m_1} \cdots p_r^{m_r})^{\alpha+\epsilon} q^{-\beta} \right),
\end{align*}
where $\alpha$ and $\beta$ are positive absolute constants of \eqref{eq220321}, and the implied constant depends only on $\epsilon$. 
\end{lemma}

\begin{proof}
By Lemma \ref{lemAB} and formula \eqref{eq072149}, we have 
\begin{align*}
&b(p_1^{m_1},f) \cdots b(p_r^{m_r},f) \\
&=\sum_{j_1=0}^{m_1} \cdots \sum_{j_r=0}^{m_r} c_{m_1}(j_1) \cdots c_{m_r}(j_r)
\sum_{n \mid \prod_{k=1}^{r} p_k^{j_k}} b_n(p_1^{j_1},\ldots,p_r^{j_r}) a(n,f). 
\end{align*}
Then it is deduced from \eqref{eq220321} that 
\begin{align*}
&\sum_{f \in B_2(q)} \omega_q(f) b(p_1^{m_1},f) \cdots b(p_r^{m_r},f) \\
&=\sum_{j_1=0}^{m_1} \cdots \sum_{j_r=0}^{m_r} c_{m_1}(j_1) \cdots c_{m_r}(j_r)
\sum_{n \mid \prod_{k=1}^{r} p_k^{j_k}} b_n(p_1^{j_1},\ldots,p_r^{j_r}) \mathbb{E}[a(n,\Theta)]
+E, 
\end{align*}
where the error term $E$ is estimated as 
\begin{align*}
E
&\ll \sum_{j_1=0}^{m_1} \cdots \sum_{j_r=0}^{m_r} |c_{m_1}(j_1)| \cdots |{c}_{m_r}(j_r)|
\sum_{n \mid \prod_{k=1}^{r} p_k^{j_k}} b_n(p_1^{j_1},\ldots,p_r^{j_r}) n^\alpha q^{-\beta} \\
&\ll d(p_1^{m_1})^3 \cdots d(p_r^{m_r})^3 (p_1^{m_1} \cdots p_r^{m_r})^\alpha q^{-\beta}
\end{align*}
by using $|c_m(j)| \leq1$ and \eqref{eq072150}. 
By an analogous argument, we obtain
\begin{align*}
&\mathbb{E}[ b(p_1^{m_1},\Theta) \cdots b(p_r^{m_r},\Theta)]\\
&=\sum_{j_1=0}^{m_1} \cdots \sum_{j_r=0}^{m_r} c_{m_1}(j_1) \cdots c_{m_r}(j_r)
\sum_{n \mid \prod_{k=1}^{r} p_k^{j_k}} b_n(p_1^{j_1},\ldots,p_r^{j_r}) \mathbb{E}[a(n,\Theta)]. 
\end{align*}
Therefore the desired result follows by the bound $d(n) \ll_\epsilon n^\epsilon$. 
\end{proof}

We apply Lemma \ref{lemAveB} to study the integral moments of $R_Y(s,f)$. 
The following proposition is an analogue of \cite[Lemma 3.5]{LamzouriLesterRadziwill2019}. 

\begin{proposition}\label{propIM}
Let $s=\sigma+it$ be a fixed complex number with $\sigma>1/2$. 
Suppose that $\Theta$ and $\omega_q$ satisfies condition \eqref{eq220321}. 
We put $L_1=2\alpha/\beta$, where $\alpha$ and $\beta$ are as in \eqref{eq220321}. 
For $Y<q$, we have 
\[
\sum_{f \in B_2(q)} \omega_q(f) \overline{R_Y(s,f)}^j R_Y(s,f)^l
=\mathbb{E}\left[ \overline{R_Y(s,\Theta)}^j R_Y(s,\Theta)^l \right]
+O\left( Y^{j+l} q^{-\beta/2} \right)
\]
for any non-negative integers $j,l$ with $j+l \leq \log{q}/(L_1\log\log{q})$, where the implied constant is absolute. 
\end{proposition}

\begin{proof}
For simplicity, we write
\[
\mathop{ \sum_{p} \sum_{m=1}^{\infty} } \limits_{p^m \leq Y} A(p^m)
=\sum_{p^m \leq Y} A(p^m)
\]
for any arithmetic function $A(n)$.  
Then we obtain
\begin{align*}
&\overline{R_Y(s,f)}^j R_Y(s,f)^l \\
&=\sum_{p_1^{m_1} \leq Y} \cdots \sum_{p_j^{m_j} \leq Y}
\sum_{q_1^{n_1} \leq Y} \cdots \sum_{q_l^{n_l} \leq Y}
\frac{b(p_1^{m_1},f) \cdots b(p_j^{m_j},f)
b(q_1^{n_1},f) \cdots b(q_l^{n_l},f)}
{p_1^{m_1 \overline{s}} \cdots p_j^{m_j \overline{s}} q_1^{n_1s} \cdots q_l^{n_ls}}, 
\end{align*}
which remains valid if we replace the symbol $f$ with $\Theta$. 
By Lemma \ref{lemAveB}, the difference is estimated as
\begin{align*}
&\sum_{f \in B_2(q)} \omega_q(f) \overline{R_Y(s,f)}^j R_Y(s,f)^l
-\mathbb{E}\left[ \overline{R_Y(s,\Theta)}^j R_Y(s,\Theta)^l \right] \\
&\ll_\epsilon Y^{(\alpha+\epsilon)(j+l)} q^{-\beta} 
\sum_{p_1^{m_1} \leq Y} \cdots \sum_{p_j^{m_j} \leq Y}
\sum_{q_1^{n_1} \leq Y} \cdots \sum_{q_l^{n_l} \leq Y}
\frac{1}{p_1^{m_1\sigma} \cdots p_j^{m_j\sigma} q_1^{n_1\sigma} \cdots q_l^{n_l\sigma}} \\
&\ll Y^{(\alpha+1)(j+l)} q^{-\beta}, 
\end{align*}
where we take $\epsilon=1/2$. 
Using the assumptions $Y<q$ and $j+l \leq \log{q}/(L_1\log\log{q})$, we have $Y^{\alpha(j+l)} < q^{\beta/2}$. 
Hence we obtain the conclusion. 
\end{proof}

By Lemma \ref{lemIM} and Proposition \ref{propIM}, we obtain the following corollary. 

\begin{corollary}\label{corIM}
Let $s=\sigma+it$ be a fixed complex number with $\sigma>1/2$. 
Suppose that $\Theta$ and $\omega_q$ satisfies \eqref{eq220323} and \eqref{eq220321}. 
Then, for $Y=(\log{q})^B$ with $B \geq1$, there exist positive constants $K_2=K_2(\sigma,B)$ and $L_2=L_2(\sigma,B)$ such that
\[
\sum_{f \in B_2(q)} \omega_q(f) |R_Y(s,f)|^{2k}
\ll
\begin{cases}
K_2^{2k} 
& \text{if $\sigma>1$}, \\
(K_2\log\log\log{q})^{2k} 
& \text{if $\sigma=1$}, \\
\displaystyle{ \left( K_2 \frac{(\log{q})^{1-\sigma}}{\log\log{q}} \right)^{2k} } 
& \text{if $1/2<\sigma<1$}, 
\end{cases}
\]
for any integer $k$ such that $1 \leq k \leq \log{q}/(L_2\log\log{q})$, where the implied constants depend only on $s$ and $B$. 
\end{corollary}

\subsection{Treatment of $g$-functions}\label{sec3.2}
In the study of the value-distributions of Dirichlet $L$-functions, Ihara--Matsumoto \cite{IharaMatsumoto2011b} considered the function
\[
g_z(s,\chi)
=\exp\left( \frac{iz}{2} \log{L}(s,\chi) \right), 
\]
where $\chi$ is a Dirichlet character and $z \in \mathbb{C}$.
Then, we define similar $g$-functions by using $R_Y(s,f)$ and $R_Y(s,\Theta)$.
For $Y<q$ and $z\in\mathbb{C}$, we define 
\[
g_z(s,f;Y)
=\exp\left(\frac{iz}{2} R_Y(s,f)\right) 
\quad\text{and}\quad
g_z(s,\Theta;Y)
=\exp\left(\frac{iz}{2} R_Y(s,\Theta)\right). 
\]
One of the key ideas of Ihara--Matsumoto was approximating $g_z(s,\chi)$ with a certain function $g_z^+(s,\chi;X)$ possessing the infinite series representation
\[
g_z^+(s,\chi;X)
=\sum_{n=1}^{\infty} \frac{\chi(n) d_z(n)}{n^s} e^{-n/X}, 
\]
where $d_z(n)$ is the $z$-th divisor function; see \cite[Proposition 2.2.1]{IharaMatsumoto2011b}.  
However, this approximation of $g_z(s,\chi)$ relies on the completely multiplicativity of the Dirichlet character $\chi(n)$. 
Thus, we study $g_z(s,f;Y)$ and $g_z(s,\Theta;Y)$ in this paper by a method slightly different from \cite{IharaMatsumoto2011b}. 
We use the Taylor series
\begin{align*}
g_z(s,f;Y)
&=\sum_{k \leq N} \frac{1}{k!} \left(\frac{iz}{2}\right)^k R_Y(s,f)^k
+\sum_{k>N} \frac{1}{k!} \left(\frac{iz}{2}\right)^k R_Y(s,f)^k \\
&=g_z^\flat(s,f;Y,N)+ g_z^\#(s,f;Y,N), 
\end{align*}
say. 
Similarly, we obtain $g_z(s,\Theta;Y)=g_z^\flat(s,\Theta;Y,N)+g_z^\#(s,\Theta;Y,N)$ by replacing $f$ with $\Theta$. 
In this subsection, we study the second moments of these functions. 

\begin{lemma}\label{lemG-}
Under the assumptions of Corollary \ref{corIM}, we take a positive integer $N$ as $N=\lfloor \log{q}/(L_2\log\log{q}) \rfloor$. 
Then there exist positive constants $a_3=a_3(\sigma,B)$ and $b_5=b_5(\sigma,B)$ such that
\begin{align*}
\sum_{f \in B_2(q) \setminus \mathcal{E}_1} \omega_q(f) |g_z^\#(s,f;Y,N)|^2
&\ll \exp\left(-b_5 \frac{\log{q}}{\log\log{q}}\right), \\
\mathbb{E}\left[|g_z^\#(s,\Theta;Y,N)|^2\right]
&\ll \exp\left(-b_5 \frac{\log{q}}{\log\log{q}}\right)
\end{align*}
for all $z \in \mathbb{C}$ with $|z| \leq a_3 R_\sigma(q)$, where $R_\sigma(q)$ and $\mathcal{E}_1=\mathcal{E}_1(q,s;Y)$ are as in the statement of Proposition \ref{propCM2}, and the implied constants depend only on $s$ and $B$. 
\end{lemma}

\begin{proof}
Applying the Cauchy--Schwarz inequality, we have 
\begin{align}\label{eq082034}
&\sum_{f \in B_2(q) \setminus \mathcal{E}_1} \omega_q(f) |g_z^\#(s,f;Y,N)|^2 \\
&\ll \sum_{j>N} \sum_{l>N} \frac{1}{j!l!}
\left(\frac{|z|}{2}\right)^{j+l}
\Bigg( \sum_{f \in B_2(q) \setminus \mathcal{E}_1} \omega_q(f)
|R_Y(s,f)|^{2(j+l)} \Bigg)^{1/2} \nonumber \\
&\ll \sum_{k>2N} \frac{|z|^k}{k!}
\Bigg( \sum_{f \in B_2(q) \setminus \mathcal{E}_1} \omega_q(f)
|R_Y(s,f)|^{2k} \Bigg)^{1/2} \nonumber
\end{align}
along with the inequality $\sum_{j+l=k} k!/(j!l!) \leq 2^k$. 
Recall that $\sum_{f \in B_2(q)} \omega_q(f)$ is uniformly bounded by \eqref{eq220321}. 
If $\sigma \geq1$, then we know
\begin{gather}\label{eq100143}
|R_Y(s,f)|
\leq \sum_{p \leq Y} \sum_{m=1}^{\infty} \frac{2}{m} p^{-m\sigma}
\ll
\begin{cases}
\log\zeta(\sigma) 
& \text{for $\sigma>1$}, \\
\log\log{Y} 
& \text{for $\sigma=1$}. 
\end{cases}
\end{gather}
Hence we deduce
\begin{gather}\label{eq082042}
|z|^k \Bigg( \sum_{f \in B_2(q) \setminus \mathcal{E}_1} \omega_q(f) |R_Y(s,f)|^{2k} \Bigg)^{1/2}
\ll \left( a_3 K \frac{\log{q}}{\log\log{q}} \right)^k
\end{gather}
for $|z| \leq a_3 R_\sigma(q)$, where $K=K(\sigma,B)$ is a positive constant, and the implied constant depends only on $s$ and $B$. 
Furthermore, \eqref{eq082042} remains valid in the case of $1/2<\sigma<1$ by the definition of $\mathcal{E}_1(q,s;Y)$. 
Then, we use the Stirling formula to derive $k! \gg (N/2)^k$ for $k>2N$. 
By \eqref{eq082042}, we obtain 
\[
\frac{|z|^k}{k!}
\Bigg( \sum_{f \in B_2(q) \setminus \mathcal{E}_1} \omega_q(f) |R_Y(s,f)|^{2k} \Bigg)^{1/2}
\ll 2^{-k}
\]
by choosing $a_3=(10KL_2)^{-1}$. 
Thus we deduce from \eqref{eq082034} that
\[
\sum_{f \in B_2(q) \setminus \mathcal{E}_1} \omega_q(f) |g_z^\#(s,f;Y,N)|^2
\ll \sum_{k>2N} 2^{-k}
\ll \exp\left(-b(\sigma,B) \frac{\log{q}}{\log\log{q}}\right)
\]
with some constant $b(\sigma,B)>0$. 
Next, we show the result for $g_z^\#(s,\Theta;Y,N)$. 
Similarly to \eqref{eq082034}, we obtain 
\[
\mathbb{E}\left[|g_z^\#(s,\Theta;Y,N)|^2\right]
\ll \sum_{k>2N} \frac{|z|^k}{k!}
\mathbb{E}\left[|R_Y(s,\Theta)|^{2k}\right]^{1/2}. 
\]
By Lemma \ref{lemIM}, the upper bounds
\[
|z|^k \mathbb{E}\left[|R_Y(s,\Theta)|^{2k}\right]^{1/2}
\ll
\begin{cases}
\displaystyle{
\left(a_3 K' \frac{\log{q}}{\log\log{q}}\right)^k
} 
& \text{for $\sigma \geq1$}, \\
\displaystyle{
\left(a_3 K' \frac{k (\log{q})^\sigma}{(k\log2{k})^\sigma}\right)^k
}
& \text{for $1/2<\sigma<1$
}
\end{cases}
\]
are valid with some positive constant $K'=K'(\sigma,B)$. 
Thus, for $k>2N$, we deduce
\[
\frac{|z|^k}{k!}
\mathbb{E}\left[|R_Y(s,\Theta)|^{2k}\right]^{1/2}
\ll 2^{-k}
\]
if we take $a_3=(20K'L_2)^{-1}$. 
Therefore we obtain
\[
\mathbb{E}\left[|g_z^\#(s,\Theta;Y,N)|^2\right]
\ll \sum_{k>2N}2^{-k}
\ll \exp\left(-b'\frac{\log{q}}{\log\log{q}}\right)
\]
with some constant $b'(\sigma,B)>0$, which completes the proof.  
\end{proof}

\begin{lemma}\label{lemG+}
Under the assumptions of Corollary \ref{corIM}, we take a positive integer $N$ as $N=\lfloor \log{q}/(L_2\log\log{q}) \rfloor$. 
For any $c>0$, there exists a positive constant $a_4=a_4(\sigma,B,c)$ such that 
\begin{align*}
\sum_{f \in B_2(q)} \omega_q(f) |g_z^\flat(s,f;Y,N)|^2
&\ll \exp\left(c \frac{\log{q}}{\log\log{q}}\right), \\
\mathbb{E}\left[|g_z^\flat(s,\Theta;Y,N)|^2\right]
&\ll \exp\left(c \frac{\log{q}}{\log\log{q}}\right)
\end{align*}
for all $z \in \mathbb{C}$ with $|z| \leq a_4 R_\sigma(q)$, where $R_\sigma(q)$ is defined as in \eqref{eq301946}. 
Here, the implied constants depend only on $s$ and $B$. 
\end{lemma}

\begin{proof}
Applying Lemma \ref{lemIM} and Corollary \ref{corIM}, we obtain 
\begin{align*}
\sum_{f  \in B_2(q)} \omega_q(f) |g_z^\flat(s,f;Y,N)|^2
&\ll \sum_{k \leq{2N}} \frac{1}{k!}
\left(a_4 K \frac{\log{q}}{\log\log{q}} \right)^k, \\
\mathbb{E}\left[|g_z^\flat(s,\Theta;Y,N)|^2\right]
&\ll \sum_{k \leq{2N}} \frac{1}{k!}
\left(a_4 K \frac{\log{q}}{\log\log{q}} \right)^k
\end{align*}
for $|z| \leq a_4 R_\sigma(q)$, where $K=K(\sigma,B)$ is a positive constant. 
We choose $a_4=c/K$, and the desired results follow directly from the Taylor series of the exponential. 
\end{proof}

\subsection{Calculations of complex moments}\label{sec3.3}
Since we have $\overline{R_Y(s,f)}=R_Y(\overline{s},f)$ and $\overline{R_Y(s,\Theta)}=R_Y(\overline{s},\Theta)$, it follows from the definitions of $g$-functions that
\begin{align*}
\psi_{z,z'}(R_Y(s,f))
&=g_z(\overline{s},f;Y) g_{z'}(s,f;Y), \\
\psi_{z,z'}(R_Y(s,\Theta))
&=g_z(\overline{s},\Theta;Y) g_{z'}(s,\Theta;Y).  
\end{align*}
For the proof of Proposition \ref{propCM2}, we estimate the difference
\begin{align*}
&\widetilde{\mathcal{M}}_{s,q}(z,z',\omega_q;Y)^{\mathcal{E}_1}
-\widetilde{\mathcal{M}}_s(z,z',\Theta;Y)\\
&=\sum_{f \in B_2(q) \setminus \mathcal{E}_1} \omega_q(f) g_z(\overline{s},f;Y) g_{z'}(s,f;Y)
-\mathbb{E}\left[g_z(\overline{s},\Theta;Y) g_{z'}(s,\Theta;Y)\right]. 
\end{align*}
We begin with considering the contributions of $g_z^\flat(s,f;Y,N)$ and $g_z^\flat(s,\Theta;Y,N)$. 

\begin{proposition}\label{propG+}
Let $s=\sigma+it$ be a fixed complex number with $\sigma>1/2$. 
Suppose that $\Theta$ and $\omega_q$ satisfies \eqref{eq220323}, \eqref{eq220321}, and \eqref{eq220322}. 
For $Y=(\log{q})^B$ with $B \geq1$, there exist positive constants $a_5=a_5(\sigma,B)$, $b_6=b_6(\sigma,B)$, and $L_3=L_3(\sigma,B)$ such that 
\begin{align*}
&\sum_{f \in B_2(q) \setminus \mathcal{E}_1} \omega_q(f)
g_z^\flat(\overline{s},f;Y,N) g_{z'}^\flat(s,f;Y,N) \\
&=\mathbb{E}\left[g_z^\flat(\overline{s},\Theta;Y,N) g_{z'}^\flat(s,\Theta;Y,N)\right] 
+O\left( \exp\left(-b_6 \frac{\log{q}}{\log\log{q}}\right) \right)
\end{align*}
for $z, z' \in \mathbb{C}$ with $\max\{|z|,|z'|\} \leq a_5 R_\sigma(q)$, where we take $N=\lfloor \log{q}/(L_3\log\log{q}) \rfloor$. 
Here, $R_\sigma(q)$ and $\mathcal{E}_1=\mathcal{E}_1(q,s;Y)$ are as in the statement of Proposition \ref{propCM2}, and the implied constant depends only on $s$ and $B$. 
\end{proposition}

\begin{proof}
Applying Proposition \ref{propIM}, we have
\[
\sum_{f \in B_2(q)} \omega_q(f)
g_z^\flat(\overline{s},f;Y,N) g_{z'}^\flat(s,f;Y,N) 
=\mathbb{E}\left[g_z^\flat(\overline{s},\Theta;Y,N) g_{z'}^\flat(s,\Theta;Y,N)\right]
+E_1,
\]
where 
\[
E_1
\ll \sum_{j \leq N} \sum_{l \leq N} \frac{1}{j!l!}
\left(\frac{|z|}{2}\right)^j \left(\frac{|z'|}{2}\right)^l
Y^{j+l} q^{-\beta/2}
\ll q^{-\beta/2} \max\{|z|,|z'|\}^{2N} Y^{2N}. 
\]
Let $L_3=L_3(\sigma,B)$ be large enough to keep the bound $\max\{|z|,|z'|\}^{2N} Y^{2N} \ll q^{\beta/4}$. 
Then we obtain $E_1 \ll q^{-\beta/4}$. 
Hence, the remaining work is giving the bound of 
\[
E_2
=\sum_{f \in \mathcal{E}_1} \omega_q(f)
g_z^\flat(\overline{s},f;Y,N) g_{z'}^\flat(s,f;Y,N).  
\]
We just consider the case of $1/2<\sigma<1$, otherwise it is the empty sum. 
Then we see that 
\begin{align*}
E_2
&=\sum_{j \leq N} \sum_{l \leq N} \frac{1}{j!l!}
\left(\frac{iz}{2}\right)^j \left(\frac{iz'}{2}\right)^l
\sum_{f \in \mathcal{E}_1} \omega_q(f) R_Y(\overline{s},f)^j R_Y(s,f)^l \\
&\ll \sum_{k \leq 2N} \frac{|z|^k}{k!}
\Bigg(\sum_{f \in \mathcal{E}_1} \omega_q(f) \Bigg)^{1/2}
\Bigg(\sum_{f \in B_2(q)} \omega_q(f) |R_Y(s,f)|^{2k}\Bigg)^{1/2}
\end{align*}
by the Cauchy--Schwarz inequality. 
Note that Lemma \ref{lemES1} yields
\[
\sum_{f \in \mathcal{E}_1} \omega_q(f)
\ll \exp\left(-\frac{b_3}{2} \frac{\log{q}}{\log\log{q}}\right)
\]
by condition \eqref{eq220322}. 
Furthermore, we deduce from Corollary \ref{corIM} the bound  
\[
\sum_{k \leq 2N} \frac{|z|^k}{k!}
\Bigg( \sum_{f \in B_2(q)} \omega_q(f) |R_Y(s,f)|^{2k} \Bigg)^{1/2}
\ll\exp\left( \frac{b_3}{4} \frac{\log{q}}{\log\log{q}} \right)
\]
by choosing $a_5=a_5(\sigma,B)$ sufficiently small, along the same line as the proof of Lemma \ref{lemG+}. 
From the above, we obtain
\begin{align*}
&\sum_{f \in B_2(q) \setminus \mathcal{E}_1} \omega_q(f)
g_z^\flat(\overline{s},f;Y,N) g_{z'}^\flat(s,f;Y,N) 
-\mathbb{E}\left[g_z^\flat(\overline{s},\Theta;Y,N) g_{z'}^\flat(s,\Theta;Y,N)\right]\\
&\qquad
\ll q^{-\beta/4}+ \exp\left(-\frac{b_3}{4} \frac{\log{q}}{\log\log{q}}\right)
\ll \exp\left(-\frac{b_3}{4} \frac{\log{q}}{\log\log{q}}\right)
\end{align*}
as desired. 
\end{proof}

\begin{proof}[Proof of Proposition \ref{propCM2}]
Let $a_2=\min\{a_3, a_4, 1\}$ with the constants $a_3=a_3(\sigma,B)$ and $a_4=a_4(\sigma,B,c)$ of Lemmas \ref{lemG-} and \ref{lemG+}, respectively. 
First, we use the Cauchy--Schwarz inequality to obtain 
\begin{align*}
&\sum_{f \in B_2(q) \setminus \mathcal{E}_1} \omega_q(f)
g_z(\overline{s},f;Y) g_{z'}(s,f;Y) 
-\sum_{f \in B_2(q) \setminus \mathcal{E}_1} \omega_q(f)
g_z^\flat(\overline{s},f;Y,N) g_{z'}^\flat(s,f;Y,N) \\
&\ll \Bigg(\sum_{f \in B_2(q)} \omega_q(f)
|g_z^\flat(\overline{s},f;Y,N)|^2\Bigg)^{1/2}
\Bigg(\sum_{f \in B_2(q) \setminus \mathcal{E}_1} \omega_q(f)
|g_{z'}^\#(s,f;Y,N)|^2\Bigg)^{1/2} \\
&+\Bigg(\sum_{f \in B_2(q) \setminus \mathcal{E}_1} \omega_q(f)
|g_z^\#(\overline{s},f;Y,N)|^2\Bigg)^{1/2}
\Bigg(\sum_{f \in B_2(q)} \omega_q(f)
|g_{z'}^\flat(s,f;Y,N)|^2\Bigg)^{1/2} \\
&+\Bigg(\sum_{f \in B_2(q) \setminus \mathcal{E}_1} \omega_q(f)
|g_z^\#(\overline{s},f;Y,N)|^2\Bigg)^{1/2}
\Bigg(\sum_{f \in B_2(q) \setminus \mathcal{E}_1} \omega_q(f)
|g_{z'}^\#(s,f;Y,N)|^2\Bigg)^{1/2}. 
\end{align*}
By Lemmas \ref{lemG-} and \ref{lemG+} with $c=b_5/2$, this is estimated as
\[
\ll \exp\left(-\frac{b_5}{8} \frac{\log{q}}{\log\log{q}}\right),
\]
where the implied constant depends only on $s$ and $B$. 
Furthermore, by a similar argument, we have
\begin{align*}
&\mathbb{E}\left[g_z(\overline{s},\Theta;Y) g_{z'}(s,\Theta;Y)\right]\\
&=\mathbb{E}\left[g_z^\flat(\overline{s},\Theta;Y,N) g_{z'}^\flat(s,\Theta;Y,N)\right]
+O\left(\exp\left(-\frac{b_5}{8} \frac{\log{q}}{\log\log{q}}\right)\right). 
\end{align*}
Hence the result is deduced from Proposition \ref{propG+}. 
\end{proof}

\begin{proof}[Proof of Proposition \ref{propCM1}]
Let $B_1=\max\left\{(2B+6)(\sigma-\frac{1}{2})^{-1}, 1\right\}$. 
Then we put 
\[
\mathcal{E}(q,s)
=\mathcal{E}_1 \cup \mathcal{E}_2
\]
for $Y=(\log{q})^{B_1}$, where $\mathcal{E}_1=\mathcal{E}_1(q,s;Y)$ and $\mathcal{E}_2=\mathcal{E}_2(q,s;Y)$ are the subsets of \eqref{eq271607} and \eqref{eq282331}, respectively. 
Note that condition \eqref{eq302036} is satisfied by Lemmas \ref{lemES1} and \ref{lemES2}. 
If $f \notin \mathcal{E}_2$, and the inequality $\max\{|z|, |z'|\} \leq \log{q}$ holds, then we have 
\[
\psi_{z,z'}(\log{L}(s,f)-R_Y(s,f))
=1+O\left((\log{q})^{-B}\right)
\]
by the definition of $\mathcal{E}_2$. 
The left-hand side of \eqref{eq302035} is calculated as
\begin{align}\label{eq011616}
&\widetilde{\mathcal{M}}_{s,q}(z,z',\omega_q)^{\mathcal{E}} \\
&=\widetilde{\mathcal{M}}_{s,q}(z,z',\omega_q;Y)^{\mathcal{E}}
+O\Bigg(\frac{1}{(\log{q})^B}
\sum_{f \in B_2(q) \setminus \mathcal{E}} \omega_q(f)
\left|\psi_{z,z'}(R_Y(s,f))\right|\Bigg). \nonumber
\end{align}
In order to get the main term, we evaluate the difference
\[
E_1
=\widetilde{\mathcal{M}}_{s,q}(z,z',\omega_q;Y)^{\mathcal{E}_1}
-\widetilde{\mathcal{M}}_{s,q}(z,z',\omega_q;Y)^{\mathcal{E}}
=\sum_{f \in \mathcal{E} \setminus \mathcal{E}_1} \omega_q(f) \psi_{z,z'}(R_Y(s,f)). 
\]
If $1/2<\sigma<1$, then it holds by the definition of $\mathcal{E}_1$ that
\[
|\psi_{z,z'}(R_Y(s,f))|
\leq \exp\left( a_1 R_\sigma(q) |R_Y(s,f)| \right)
\leq \exp\left( a_1 \frac{\log{q}}{\log\log{q}} \right)
\]
for $f \notin \mathcal{E}_1$ and $\max\{|z|, |z'|\} \leq a_1 R_\sigma(q)$. 
If $\sigma \geq1$, then bounds \eqref{eq100143} are available for all $f \in B_2(q)$. 
Hence we obtain 
\[
|\psi_{z,z'}(R_Y(s,f))|
\leq \exp\left( a_1 K \frac{\log{q}}{\log\log{q}} \right)
\]
for $\sigma \geq1$ with a constant $K=K(\sigma, B) \geq1$. 
Taking $a_1$ with $a_1 \leq b_1/(4K)$, we obtain
\[
E_1
\ll \exp\left(\frac{b_1}{4} \frac{\log{q}}{\log\log{q}}\right)
\sum_{f \in \mathcal{E}} \omega_q(f)
\ll \exp\left(-\frac{b_1}{4} \frac{\log{q}}{\log\log{q}}\right) 
\]
as a consequence of \eqref{eq220322} and \eqref{eq302036}. 
Thus the formula 
\begin{gather}\label{eq100155}
\widetilde{\mathcal{M}}_{s,q}(z,z',\omega_q;Y)^{\mathcal{E}} 
=\widetilde{\mathcal{M}}_{s,q}(z,z',\omega_q;Y)^{\mathcal{E}_1}
+O\left(\exp\left(-\frac{b_1}{4} \frac{\log{q}}{\log\log{q}}\right)\right) 
\end{gather}
follows. 
Next, applying Proposition \ref{propCM2} and Lemma \ref{lemM}, we have 
\begin{gather}\label{eq100156}
\widetilde{\mathcal{M}}_{s,q}(z,z',\omega_q;Y)^{\mathcal{E}_1}
=\widetilde{\mathcal{M}}_s(z,z',\Theta)
+E_2, 
\end{gather}
where the error term is estimated as
\begin{align*}
E_2
&\ll \exp\left(-b_2 \frac{\log{q}}{\log\log{q}}\right)
+(1+E_s(z,z',\Theta)) (\log{q})^{-B-2} \\
&\ll \frac{1+E_s(z,z',\Theta)}{(\log{q})^B}. 
\end{align*}
Finally, we evaluate the error term in the right-hand side of \eqref{eq011616}. 
Let $(\xi,\xi')$ be the pair of \eqref{eq092258}. 
Then we see that
\[
\sum_{f \in B_2(q) \setminus \mathcal{E}} \omega_q(f)
\left|\psi_{z,z'}(R_Y(s,f))\right|
\leq \sum_{f \in B_2(q) \setminus \mathcal{E}_1} \omega_q(f)
\psi_{\xi,\xi'}(R_Y(s,f)) 
\]
since $|\psi_{z,z'}(w)|=\psi_{\xi,\xi'}(w)$ and $\mathcal{E} \supset \mathcal{E}_1$. 
Applying Proposition \ref{propCM2} and Lemma \ref{lemM} again, we obtain
\begin{align}\label{eq100219}
\sum_{f \in B_2(q) \setminus \mathcal{E}} \omega_q(f)
\left|\psi_{z,z'}(R_Y(s,f))\right|
&\leq \widetilde{\mathcal{M}}_s(\xi,\xi',\Theta)
+O\left(\frac{1+E_s(\xi,\xi',\Theta)}{(\log{q})^B}\right)\\
&\ll 1+{E}_s(z,z',\Theta). \nonumber
\end{align}
Then the desired result follows by formulas \eqref{eq011616}, \eqref{eq100155}, \eqref{eq100156}, and \eqref{eq100219}. 
\end{proof}

We obtain the following corollary, which is used in the proof of Theorem \ref{thmMV}. 

\begin{corollary}\label{corUB}
Under the assumptions of Proposition \ref{propCM1}, we obtain
\begin{align*}
\sum_{f \in B_2(q) \setminus \mathcal{E}} \omega_q(f)
\exp\left(a |\RE\log{L}(s,f)|\right)
&\ll 1, \\
\sum_{f \in B_2(q) \setminus \mathcal{E}} \omega_q(f)
\exp\left(a |\IM\log{L}(s,f)|\right) 
&\ll 1 
\end{align*} 
for any fixed real number $a \geq0$. 
\end{corollary}

\begin{proof}
We prove just the first estimate since the second one is proved in a similar way. 
Let $b \in \mathbb{R}$. 
Using Proposition \ref{propCM1} with $z=z'=-ib$, we derive
\begin{align*}
&\sum_{f\in{B}_2(q)\setminus\mathcal{E}}\omega_q(f)
\exp\left(b \RE \log{L}(s,f)\right) \\
&=\mathbb{E}\left[ \exp\left(b \RE \log{L}(s,\Theta)\right) \right]
+O\left(\frac{1+ \mathbb{E}\left[ \exp\left(b \RE \log{L}(s,\Theta)\right) \right]}{(\log{q})^B}\right) \\
&\ll \mathbb{E}\left[ \exp\left(|b| |\RE \log{L}(s,\Theta)\right)| \right]. 
\end{align*}
Therefore, we deduce from Lemma \ref{lemMGF} that
\[
\sum_{f\in{B}_2(q)\setminus\mathcal{E}}\omega_q(f)
\exp\left(b \RE \log{L}(s,f)\right)
\ll 1. 
\]
Hence we obtain the result since $e^{a|x|} \leq e^{ax}+e^{-ax}$ holds. 
\end{proof}

\section{Completion of the proofs}\label{sec4}

\subsection{Proof of Theorem \ref{thmPDF}}\label{sec4.1}
Let $\Theta=(\Theta_p)$ be a sequence of $[0,\pi]$-valued random variables such that condition \eqref{eq220319} holds. 
For $s=\sigma+it \notin \mathbb{R}$, we define 
\begin{gather}\label{eq111343}
\mu_s(A;\Theta)
=\mathbb{P}(\log{L}(s,\Theta) \in A), 
\quad 
A \in \mathcal{B}(\mathbb{C}). 
\end{gather}
On the other hand, we see that the random variable $L(s,\Theta)$ is real-valued in the case $s=\sigma \in \mathbb{R}$. 
Thus, for $s=\sigma \in \mathbb{R}$, we also define 
\begin{gather}\label{eq111345}
\mu_\sigma(A;\Theta)
=\mathbb{P}(\log{L}(\sigma,\Theta) \in A), 
\quad 
A \in \mathcal{B}(\mathbb{R}). 
\end{gather}
Then $\mu_s$ and $\mu_\sigma$ are probability measures on $(\mathbb{C}, \mathcal{B}(\mathbb{C}))$ and $(\mathbb{R}, \mathcal{B}(\mathbb{R}))$, respectively. 
We prove that they have continuous density functions by L\'{e}vy's inversion formula. 

\begin{lemma}[Jessen--Wintner \cite{JessenWintner1935}, Section 3]\label{lemLevy}
Let $\mu$ be a probability measure on $(\mathbb{R}^d, \mathcal{B}(\mathbb{R}^d))$ with the characteristic function  
\[
\Lambda(y;\mu)
=\int_{\mathbb{R}^d} e^{i \langle x, y \rangle} \,d \mu(x),  
\]
where $\langle x, y \rangle$ is the standard inner product on $\mathbb{R}^d$. 
For an integer $n \geq0$, we suppose 
\begin{gather}\label{eq102200}
\Lambda(y;\mu)
=O\left( |y|^{-(d+n+\eta)} \right)
\end{gather}
with some constant $\eta>0$. 
Then $\mu$ is absolutely continuous, and the Radon--Nikod\'{y}m derivative is determined by
\[
D_\mu(x)
=\frac{1}{(2\pi)^d} \int_{\mathbb{R}^d}
\Lambda(y;\mu) e^{-i\langle x, y \rangle} \,dy,  
\]
which is a non-negative $C^n$-function on $\mathbb{R}^d$. 
\end{lemma}

\begin{proof}[Proof of Theorem \ref{thmPDF}]
For the proof of the statement $(\mathrm{i})$, we identify $\mathbb{C}$ with $\mathbb{R}^2$ by the map $\iota: x+iy \mapsto (x,y)$. 
Let $s=\sigma+it$ be a fixed complex number with $\sigma>1/2$ and $t \neq0$. 
Then we have $\Lambda(z; \mu_s)=\widetilde{\mathcal{M}}_s(z,\overline{z},\Theta)$ for all $z \in \mathbb{C}$ by definition. 
Hence Proposition \ref{propUB} $(\mathrm{i})$ yields that condition \eqref{eq102200} holds for any $n \geq0$. 
Then we put 
\[
\mathcal{M}_s(w,\Theta)
=2 \pi D_{\mu_s}(u,v) 
\]
for $w=u+iv$, which satisfies the equality
\[
\mathbb{P}(\log{L}(s,\Theta) \in A)
=\int_{\iota(A)} D_{\mu_s}(u,v) \,dudv
=\int_{A} \mathcal{M}_s(w,\Theta) \,|dw|
\]
as desired. 
In the case $s=\sigma>1/2$, we use Proposition \ref{propUB} $(\mathrm{ii})$ to deduce \eqref{eq102200} with $d=1$. 
Then we obtain the desired result by putting $\mathcal{M}_\sigma(u,\Theta)=\sqrt{2\pi} D_{\mu_\sigma}(u)$. 
\end{proof}

Theorem \ref{thmPDF} contains further information on $\mathcal{M}_s(\,\cdot\,,\Theta)$. 
In fact, it is represented as an infinite convolution of Schwartz distributions as follows. 
Let $s=\sigma+it$ be a fixed complex number with $\sigma>1/2$ and $t \neq0$. 
For every prime number $p$, we define a Schwartz distribution $\mathcal{M}_{s,p}(w,\Theta)$ on $\mathbb{C}$ as 
\[
\int_{\mathbb{C}} \Phi(w) \mathcal{M}_{s,p}(w,\Theta) \,|dw|
=\mathbb{E}\left[ \Phi\left( -\log(1- 2(\cos \Theta_p)p^{-s}+ p^{-2s}) \right) \right]. 
\]
Denote by $p_n$ the $n$-th prime number, and put $N=\pi(Y)$ with a real number $Y \geq2$. 
By the independence of $\Theta=(\Theta_p)$, we obtain
\[
\int_{\mathbb{C}} \Phi(w) \left( \mathcal{M}_{s,p_1}(w,\Theta) * \cdots
* \mathcal{M}_{s,p_N}(w,\Theta) \right) \,|dw| 
=\mathbb{E}\left[\Phi\left( \log{L}_Y(s,\Theta) \right)\right], 
\]
where $L_Y(s,\Theta)$ is the partial Euler product of \eqref{eq110122}. 
Recall that $L_Y(s,\Theta)$ converges to $L(s,\Theta)$ as $Y \to\infty$ in law. 
Hence, Theorem \ref{thmPDF} $(\mathrm{i})$ implies the formula
\begin{gather}\label{eq111725}
\mathcal{M}_s(w,\Theta)
=\lim_{N \to\infty}
\mathcal{M}_{s,p_1}(w,\Theta) * \cdots * \mathcal{M}_{s,p_N}(w,\Theta). 
\end{gather}
Moreover, this is valid in the case $s=\sigma>1/2$ if we define $\mathcal{M}_{\sigma,p}(u,\Theta)$ as a Schwartz distribution on $\mathbb{R}$ such that 
\[
\int_{\mathbb{R}} \Phi(u) \mathcal{M}_{\sigma,p}(u,\Theta) \,|du|
=\mathbb{E}\left[\Phi\left( -\log(1-2(\cos \Theta_p) p^{-\sigma}+ p^{-2\sigma}) \right)\right]. 
\]
Finally, we consider the relation between $\mathcal{M}_{s,p}(w,\Theta)$ and $\mathcal{M}_{\sigma,p}(u,\Theta)$ in terms of Schwartz distributions. 
Note that $L(\sigma+it, \Theta) \to {L}(\sigma,\Theta)$ as $t \to0$ in law. 
Thus it is deduced from Theorem \ref{thmPDF} that
\[
\lim_{t \to0} \int_{\mathbb{C}} \Phi(w) \mathcal{M}_{\sigma+it}(w,\Theta) \,|dw|
=\int_{\mathbb{R}} \Phi(u) \mathcal{M}_{\sigma}(u,\Theta) \,|du|.  
\]
Therefore, we conclude that
\begin{gather}\label{eq300000}
\lim_{t \to0} \mathcal{M}_{\sigma+it}(u+iv,\Theta)
=\mathcal{M}_{\sigma}(u,\Theta) \delta(v), 
\end{gather}
where $\delta(v)$ is the Dirac delta distribution normalized as 
\[
\int_{\mathbb{R}} \Phi(v) \delta(v) \,|dv|
=\Phi(0). 
\]

\begin{corollary}\label{corSupp}
Let $\Theta=(\Theta_p)$ be as in Theorem \ref{thmPDF}. 
If $s=\sigma+it$ is a fixed complex number with $\sigma>1$, then the function $\mathcal{M}_s(\,\cdot\,,\Theta)$ is compactly supported on $\mathbb{C}$ for $s \notin \mathbb{R}$, and it is compactly supported on $\mathbb{R}$ for $s \in \mathbb{R}$. 
\end{corollary}

\begin{proof}
In general, we denote by $\lim_{n \to\infty} (A_1+\cdots+A_n)$ the set of points which may be represented in at least one way as the limit of $a_1+\cdots+a_n$, where $a_j$ belongs to $A_j$ for each $j$. 
Then we deduce from \eqref{eq111725} that
\[
\supp \mathcal{M}_s(\,\cdot\,,\Theta)
=\lim_{n \to\infty} \left(\supp\mathcal{M}_{s,p_1}(\,\cdot\,,\Theta)
+\cdots+\supp \mathcal{M}_{s,p_n}(\,\cdot\,,\Theta)\right) 
\]
by \cite[Theorem 3]{JessenWintner1935}. 
The support of $\mathcal{M}_{s,p}(\,\cdot\,,\Theta)$ is determined by
\[
\supp \mathcal{M}_{s,p}(\,\cdot\,,\Theta)
=\left\{ -\log(1- 2(\cos \theta) p^{-s}+ p^{-2s}) ~\middle|~ 0 \leq \theta \leq 2\pi \right\}. 
\]
Hence we have $|z| \leq 2\log\zeta(\sigma)$ if $z$ belongs to the support of $\mathcal{M}_{s,p}(\,\cdot\,,\Theta)$. 
\end{proof}

\subsection{Proof of Theorem \ref{thmMV}}\label{sec4.2}
We define
\begin{gather}\label{eq111344}
\mu_{s,q}(A; \omega_q)
=\frac{1}{N(q,s)}
\sum_{f \in B'_2(q,s)} \omega_q(f) 1_A(\log{L}(s,f)), 
\quad 
A \in \mathcal{B}(\mathbb{C})
\end{gather}
for $s=\sigma+it \notin \mathbb{R}$, and 
\[
\mu_{\sigma,q}(A; \omega_q)
=\frac{1}{N(q,\sigma)}
\sum_{f \in B'_2(q,\sigma)} \omega_q(f) 1_A(\log{L}(\sigma,f)), 
\quad 
A \in \mathcal{B}(\mathbb{R})
\]
for $s=\sigma \in \mathbb{R}$, where we put $N(q,s)=\sum_{f \in B'_2(q,s)} \omega_q(f)$.  
Then, using the zero density estimate \cite[Theorem 4]{KowalskiMichel1999} along with \eqref{eq220321} and \eqref{eq220322}, we obtain
\[
N(q,s)
=\sum_{f \in B_2(q,s)} \omega_q(f)
+O\Bigg( \sum_{f \in B_2(q)} \omega_q(f) N(f;\sigma,t-1,t+1)\Bigg)
=1+o(1)
\]
as $q\to\infty$. 
Note that $\mu_{s,q}$ and $\mu_{\sigma,q}$ are probability measures on $(\mathbb{C}, \mathcal{B}(\mathbb{C}))$ and on $(\mathbb{R}, \mathcal{B}(\mathbb{R}))$, respectively. 
To begin with, we recall the continuity theorem for probability measures. 

\begin{lemma}[Jessen--Wintner \cite{JessenWintner1935}, Section 3]\label{lemConti}
Let $(\mu_n)_{n \in \mathbb{N}}$ be a sequence of probability measures on $(\mathbb{R}^d, \mathcal{B}(\mathbb{R}^d))$, and let $\mu$ be a probability measure on $(\mathbb{R}^d, \mathcal{B}(\mathbb{R}^d))$. 
If the characteristic functions satisfy
\begin{gather}\label{eq111431}
\lim_{n \to\infty} \Lambda(y;\mu_n)
=\Lambda(y;\mu) 
\end{gather}
uniformly in every sphere $|y| \leq R$, then the limit formula
\[
\lim_{n \to\infty} \mu_n(A)
=\mu(A)
\]
holds for all $A \in \mathcal{B}(\mathbb{R}^d)$ with $\mu(\partial{A})=0$. 
\end{lemma}

In addition, we use a lemma of Ihara--Matsumoto \cite{IharaMatsumoto2011b}. 
Let $\Lambda(\mathbb{R}^d)$ be the set of all functions $\Phi \in L^1(\mathbb{R}^d)$ such that $\Phi$ is continuous and its Fourier transform belongs to $L^1(\mathbb{R}^d)$. 
Then a function $M: \mathbb{R}^d \to \mathbb{R}_{\geq0}$ is called a \textit{good density function} if it belongs to $\Lambda(\mathbb{R}^d)$ and satisfies
\[
\int_{\mathbb{R}^d} M(x) \,|dx|
=1,  
\]
where we write $|dx|=(2\pi)^{-d/2} dx_1 \cdots dx_d$ for $x=(x_1,\ldots,x_d)$. 

\begin{lemma}[Ihara--Matsumoto \cite{IharaMatsumoto2011b}, Section 5]\label{lemIhaMat}
Let $M$ be a good density function on $\mathbb{R}^d$. 
Let $(X_n)_{n \in \mathbb{N}}$ be a sequence of finite sets equipped with probability measures $\omega_n$ and functions $\ell_n: X_n \to \mathbb{R}^d$. 
Then we consider the condition 
\begin{gather}\label{eq132207}
\lim_{n \to\infty} \sum_{\chi \in X_n} \omega_n(\chi) \Phi(\ell_n(\chi))
=\int_{\mathbb{R}^d} \Phi(x) M(x) \,|dx| 
\end{gather}
for a test function $\Phi: \mathbb{R}^d \to \mathbb{C}$. 
Suppose that condition \eqref{eq132207} holds for any additive characters $\Phi(x)=e^{i \langle x, y \rangle}$ with $y \in \mathbb{R}^d$, and that the convergence is uniform in every sphere $|x| \leq R$. 
Then we have the following results. 
\begin{enumerate}
\item[$(\mathrm{a})$] 
\eqref{eq132207} holds for any bounded continuous function on $\mathbb{R}^d$. 
\item[$(\mathrm{b})$] 
\eqref{eq132207} holds for any $\Phi \in C(\mathbb{R}^d)$ such that $\Phi(x) \ll \phi_0(|x|)$, where $\phi_0(r)$ is a continuous non-decreasing function on $[0,\infty)$ which satisfies $\phi_0(r)>0$, $\phi_0(r) \to\infty$ as $r \to\infty$, and 
\begin{align}
&\sum_{\chi \in X_n} \omega_n(\chi) \phi_0(|\ell_n(\chi)|)^2
\ll 1, 
\label{eq112209}\\
&\int_{\mathbb{R}^d} \phi_0(|x|) M(x) \,|dx|
<\infty. 
\label{eq112210}
\end{align}
\end{enumerate}
\end{lemma}

\begin{remark}\label{rem:IhaMat}
The statement of Lemma \ref{lemIhaMat} is slightly different from the original one by Ihara--Matsumoto \cite{IharaMatsumoto2011b}. 
Indeed, it contained the following result in addition to $(\mathrm{a})$ and $(\mathrm{b})$ above:  
\begin{enumerate}
\item[$(\mathrm{c})$] 
\eqref{eq132207} further holds for the indicator function of either a compact subset of $\mathbb{R}^d$ or the complement of such a subset.
\end{enumerate}
However, the author was informed from Matsumoto that the proof of $(\mathrm{c})$ does not work in special cases such as fat Cantor sets. 
For more details, see the appendix in the forthcoming paper of Matsumoto \cite{Matsumoto2021+}. 
\end{remark}

\begin{proof}[Proof of Theorem \ref{thmMV} $(\mathrm{i})$]
Let $R>0$ be a fixed real number. 
First, we remark that Proposition \ref{propCM1} implies 
\begin{gather}\label{eq111418}
\sum_{f \in B'_2(q,s)} \omega_q(f) \psi_{z,\overline{z}}(\log{L}(s,f))
=\widetilde{\mathcal{M}}_s(z,\overline{z},\Theta)
+O\left((\log{q})^{-B}\right) 
\end{gather}
uniformly for $|z| \leq R$ along with \eqref{eq220322}, \eqref{eq302036}, and $|\psi_{z,\overline{z}}(w)|=1$. 
Let $(q_n)_{n \in \mathbb{N}}$ be a sequence of prime numbers such that $q_n \to\infty$ as $n \to\infty$. 
We apply Lemma \ref{lemLevy} with $\mu_n=\mu_{s,q_n}$ of \eqref{eq111344} and $\mu=\mu_s$ of \eqref{eq111343}. 
Since we have 
\begin{align*}
\Lambda(z;\mu_n)
&=\frac{1}{N(q_n,s)}
\sum_{f \in B'_2(q_n,s)} \omega_{q_n}(f) \psi_{z,\overline{z}}(\log{L}(s,f)), \\
\Lambda(z;\mu)
&=\widetilde{\mathcal{M}}_s(z,\overline{z},\Theta), 
\end{align*}
and $N(q_n,s) \to1$ as $n \to\infty$, it is deduced from \eqref{eq111418} that condition \eqref{eq111431} holds uniformly in $|z| \leq R$. 
By Theorem \ref{thmPDF} and Lemma \ref{lemLevy}, we obtain desired formula \eqref{eq111435} for $\Phi=1_A$ with any continuity set $A \subset \mathbb{C}$. 

Next, we use Lemma \ref{lemIhaMat} to see the remaining cases.  
Note that $\mathcal{M}_s(w,\Theta)$ is a good density function on $\mathbb{C}$. 
Let $X_n=B'_2(q_n,s)$. 
We define
\[
\omega_n(f)
=\frac{\omega_{q_n}(f)}{N(q_n,s)}
\quad\text{and}\quad
\ell_n(f)
=\log{L}(s,f) 
\]
for $f \in X_n$. 
From Theorem \ref{thmPDF}, we have
\[
\widetilde{\mathcal{M}}_s(z,\overline{z},\Theta)
=\int_{\mathbb{C}} \psi_{z,\overline{z}}(w) \mathcal{M}_s(w,\Theta) \,|dw|. 
\]
By this and \eqref{eq111418}, condition \eqref{eq132207} holds for $\Phi(w)=e^{i \langle z, w \rangle}$ with all $z \in \mathbb{C}$, and the convergence is uniform in $|z| \leq R$. 
Therefore we derive from Lemma \ref{lemIhaMat} $(\mathrm{a})$ that \eqref{eq111435} holds for any bounded continuous function $\Phi$. 
In order to use Lemma \ref{lemIhaMat} $(\mathrm{b})$, we proceed to checking conditions \eqref{eq112209} and \eqref{eq112210} as follows. 
\begin{itemize}
\item
Let $\sigma>1$. 
For any $\Phi \in C(\mathbb{C})$, we take the function $\phi_0$ as
\[
\phi_0(r)
=\max_{|x| \leq r} |\Phi(x)|. 
\]
Then \eqref{eq112209} is valid since we have $\phi_0(|\log{L}(s,f)|) \leq \phi_0(2\log\zeta(\sigma))$ for $\sigma>1$. 
Furthermore, \eqref{eq112210} is also valid by Corollary \ref{corSupp}. 
\item
Let $\sigma=1$. 
In this case, we take $\phi_0(r)=r^a$, where $a>0$ is a fixed real number.  
Let $\mathcal{E}=\mathcal{E}(q,s)$ be the set of Proposition \ref{propCM1}. 
Then we have 
\[
\sum_{f \in B_2(q) \setminus \mathcal{E}} \omega_q(f)
\phi_0(|\log{L}(s,f)|)
\ll 1
\]
by Corollary \ref{corUB}. 
Hence condition \eqref{eq112209} holds if we derive the estimate
\begin{gather}\label{eq171732}
\sum_{f \in \mathcal{E}} \omega_q(f)
\phi_0(|\log{L}(s,f)|)
\ll 1. 
\end{gather}
By a standard argument along with the zero-free region of $L(s,f)$, there exists an absolute constant $c>0$ such that we have
\[
\frac{L'}{L}(s,f)
\ll (\log{q}(|t|+3))^2
\]
for all $s=\sigma+it$ such that $\sigma \geq 1-c/\log{q}(|t|+2)$; see \cite[Proposition 5.7]{IwaniecKowalski2004}. 
It yields the upper bound 
\begin{gather}\label{eq112355}
\log{L}(s,f)
\ll (\log{q}(|t|+3))^2, 
\end{gather}
on the vertical line $\RE(s)=1$. 
Thus estimate \eqref{eq171732} is derived from \eqref{eq220322} and \eqref{eq302036}. 
On the other hand, we see that the expected value 
\[
\mathbb{E}[\phi(|\log{L}(s,\Theta)|)]
\]
is finite for all $\phi$ with $\phi(r) \ll e^{ar}$ by using Lemma \ref{lemMGF} and the Cauchy--Schwarz inequality. 
Hence condition \eqref{eq112210} follows in this case. 
\item
Let $1/2<\sigma\leq1$ and assume GRH. 
We take the function $\phi_0$ as $\phi_0(r)=e^{ar}$, where $a>0$ is a fixed real number.  
Then we use Corollary \ref{corUB} to derive 
\[
\sum_{f \in B_2(q) \setminus \mathcal{E}} \omega_q(f)
\phi_0(|\log{L}(s,f)|)
\ll 1
\]
in this case. 
Furthermore, we have the upper bound 
\[
\log{L}(s,f)
\ll \frac{(\log{q})^{2-2\sigma}}{\log\log{q}}
+\log\log{q}
\]
for $f \in B_2(q)$ under GRH; see \cite[Theorem 5.19]{IwaniecKowalski2004}. 
Hence we deduce
\[
\sum_{f \in \mathcal{E}} \omega_q(f)
\phi_0(|\log{L}(s,f)|)
\ll 1 
\]
along with \eqref{eq220322} and \eqref{eq302036}. 
Therefore we obtain \eqref{eq112209} in this case. 
As we have checked, condition \eqref{eq112210} is valid for $\phi_0(r)=e^{ar}$. 
\end{itemize}
From the above, Lemma \ref{lemIhaMat} $(\mathrm{b})$ yields \eqref{eq111435} for continuous test functions $\Phi$ as in the statement of Theorem \ref{thmMV} $(\mathrm{i})$. 
\end{proof}

\begin{proof}[Proof of Theorem \ref{thmMV} $(\mathrm{ii})$]
The difference from Theorem \ref{thmMV} $(\mathrm{i})$ arises from only the case $\sigma=1$. 
We prove the result in this case by using Lemma \ref{lemIhaMat} $(\mathrm{b})$ again. 
Then the estimates of Cogdell--Michel \cite[Lemmas 4.1 and 4.2]{CogdellMichel2004} yield
\[
\log{L}(1,f)
\ll\log\log{q}
\]
unconditionally. 
Hence we deduce condition \eqref{eq112209} with $\phi_0(r)=e^{ar}$ by using this instead of \eqref{eq112355}. 
Since \eqref{eq112210} is still valid, we conclude that desired formula \eqref{eq111435} holds for any $\Phi\in{C}^{\exp}(\mathbb{R})$. 
\end{proof}

\subsection{Proof of Theorem \ref{thmDB}}\label{sec4.3}
We prove Theorem \ref{thmDB} by using Esseen's inequality and its two-dimensional analogue. 

\begin{lemma}[Lo\`eve \cite{Loeve1977}]\label{lemEsseen1}
Let $\mu$ and $\nu$ be probability measures on $(\mathbb{R}, \mathcal{B}(\mathbb{R}))$ with distribution functions
\[
F(x)
=\mu((-\infty,x])
\quad\text{and}\quad
G(x)
=\nu((-\infty,x]). 
\] 
Suppose that $G$ is differentiable, and we put $A=\sup_{x \in \mathbb{R}} |G'(x)|$. 
Denote by $f(u)=\Lambda(u;\mu)$ and $g(u)=\Lambda(u;\nu)$ the characteristic functions. 
Then we have 
\begin{gather}\label{eq122329}
\sup_{x \in \mathbb{R}} |F(x)-G(x)|
\leq \frac{1}{\pi}
\int_{-R}^{R} \left| \frac{f(u)-g(u)}{u} \right| \,du
+\frac{24}{\pi} \frac{A}{R}
\end{gather}
for any $R>0$. 
\end{lemma}

\begin{lemma}[Sadikova \cite{Sadikova1966}]\label{lemEsseen2}
Let $\mu$ and $\nu$ be probability measures on $(\mathbb{R}^2, \mathcal{B}(\mathbb{R}^2))$ with distribution functions
\[
F(x,y)
=\mu((-\infty,x]\times(-\infty,y])
\quad\text{and}\quad
G(s,t)
=\nu((-\infty,x]\times(-\infty,y]). 
\] 
Suppose that $G$ is partially differentiable, and we put $A_1=\sup_{(x,y) \in \mathbb{R}^2} |G_x(x,y)|$ and $A_2=\sup_{(x,y) \in \mathbb{R}^2} |G_y(x,y)|$. 
Denote by $f(u,v)=\Lambda(u,v;\mu)$ and $g(u,v)=\Lambda(u,v;\nu)$ the characteristic functions, and we define
\[
\hat{f}(u,v)
=f(u,v)-f(u,0)f(0,v)
\quad\text{and}\quad
\hat{g}(u,v)
=g(u,v)-g(u,0)g(0,v).
\]
Then we have 
\begin{align}\label{eq121643}
&\sup_{(x,y) \in \mathbb{R}^2} |F(x,y)-G(x,y)|
\leq \frac{2}{(2\pi)^2} \iint_{[-R,R]^2}
\left| \frac{\hat{f}(u,v)-\hat{g}(u,v)}{uv} \right| \,dudv \\
&\qquad
+\frac{2}{\pi} \int_{-R}^{R} 
\left| \frac{f(u,0)-g(u,0)}{u} \right| \,du
+\frac{2}{\pi} \int_{-R}^{R} 
\left| \frac{f(0,v)-g(0,v)}{v} \right| \,dv \nonumber\\
&\qquad
+\left( 3\sqrt{2}+ 4\sqrt{3}+ \frac{24}{\pi} \right) 
\frac{2(A_1+A_2)}{R} \nonumber
\end{align}
for any $R>0$. 
\end{lemma}

\begin{proof}[Proof of Theorem \ref{thmDB} $(\mathrm{i})$]
Let $\mathcal{E}=\mathcal{E}(q,s)$ be the set of Proposition \ref{propCM1}. 
We define a probability measure $\mu_{s,q}(\,\cdot\,;\omega_q)^\mathcal{E}$ on $(\mathbb{C},\mathcal{B}(\mathbb{C}))$ by
\[
\mu_{s,q}(A;\omega_q)^\mathcal{E}
=\frac{1}{N(q,s)^\mathcal{E}}
\sum_{f \in B_2(q) \setminus \mathcal{E}} \omega_q(f) 1_A(\log{L}(s,f)), 
\]
where $N(q,s)^\mathcal{E}=\sum_{f \in B_2(q) \setminus \mathcal{E}} \omega_q(f)$. 
Then we apply Lemma \ref{lemEsseen2} with $\mu=\mu_{s,q}^\mathcal{E}$ and $\nu=\mu_s$ of \eqref{eq111343} by identifying $\mathbb{C}$ with $\mathbb{R}^2$. 
Since we have 
\[
G(x,y)
=\int_{\RE(w) \leq x,~ \IM(w) \leq y} \mathcal{M}_s(w,\Theta) \,|dw| 
\]
in this case, the function $G$ is partially differentiable. 
In particular, we obtain
\[
G_x(x,y)
=\int_{-\infty}^{y} \mathcal{M}_s(x+iv,\Theta) \,|dv| 
\quad\text{and}\quad
G_y(x,y)
=\int_{-\infty}^{x} \mathcal{M}_s(u+iy,\Theta) \,|du|.
\] 
Therefore, $|G_x(x,y)|$ is bounded as 
\[
|G_x(x,y)|
\leq G_1(x)
:=\int_{-\infty}^{\infty} \mathcal{M}_s(x+iv,\Theta) \,|dv| 
\]
for any $(x,y) \in \mathbb{R}^2$. 
Since $G_1(x)$ is a non-negative continuous function satisfying
\[
\int_{-\infty}^{\infty} G_1(x) \,|dx|
=\int_{\mathbb{C}} \mathcal{M}_s(w,\Theta) \,|dw|
=1, 
\]
we conclude that $\sup_{x \in \mathbb{R}} G_1(x)$ is finite. 
As a result, $A_1=\sup_{(x,y) \in \mathbb{R}^2} |G_x(x,y)|$ is also finite. 
One can show the finiteness of $A_2=\sup_{(x,y) \in \mathbb{R}^2} |G_y(x,y)|$ in a similar way. 
Next, we consider the characteristic functions of $\mu=\mu_{s,q}^\mathcal{E}$ and $\nu=\mu_s$. 
Note that we have
\[
N(q,s)^\mathcal{E}
=\sum_{f \in B_2(q)} \omega_q(f)
-\sum_{f \in \mathcal{E}} \omega_q(f)
=1+O\left( \exp\left(-\frac{b_1}{2} \frac{\log{q}}{\log\log{q}}\right) \right)
\]
by \eqref{eq220321}, \eqref{eq220322}, and \eqref{eq302036}. 
Let $z=u+iv$. 
Then the characteristic function of $\mu$ is calculated as
\begin{align}\label{eq122113}
f(u,v)
&=\frac{1}{N(q,s)^\mathcal{E}}
\sum_{f \in B_2(q) \setminus \mathcal{E}} \omega_q(f)
\psi_{z,\overline{z}}(\log{L}(s,f))  \\
&=\widetilde{\mathcal{M}}_{s,q}(z,\overline{z},\omega_q)
+O\left( \exp\left(-\frac{b_1}{2} \frac{\log{q}}{\log\log{q}}\right) \right) \nonumber
\end{align}
by noting that $|\psi_{z,\overline{z}}(w)|=1$. 
Furthermore, we have
\begin{gather}\label{eq122114}
g(u,v)
=\widetilde{\mathcal{M}}_s(z,\overline{z},\Theta). 
\end{gather}
Put $r=(\log{q})^{-2}$. 
We divide the first integral of the right-hand side of \eqref{eq121643} as 
\[
\iint_{[-R,R]^2} \left|\frac{\hat{f}(u,v)-\hat{g}(u,v)}{uv}\right| \,dudv
=\iint_{[-R,R]^2\setminus{C}(r)}+ \iint_{C(r)}
=I_1+I_2,  
\]
say, where we define
\[
C(r)
=\left\{ (u,v) \in [-R,R]^2 ~\middle|~ \text{$|u| \leq r$ or $|v| \leq r$} \right\}. 
\]
Then we have 
\begin{gather}\label{eq122116}
I_1
\ll \left(\log\frac{R}{r}\right)^2
\sup_{(u,v) \in [-R,R]^2} \left|\hat{f}(u,v)-\hat{g}(u,v)\right|, 
\end{gather}
and moreover, the inequality
\[
\left|\hat{f}(u,v)-\hat{g}(u,v)\right|
\leq |f(u,v)-g(u,v)|+|f(u,0)-g(u,0)|+|f(0,v)-g(0,v)|
\]
holds. 
We take $R=\frac{1}{\sqrt{2}} a_1 R_\sigma(q)$, where $a_1$ and $R_\sigma(q)$ are as in Proposition \ref{propCM1}. 
Then the condition $|z| \leq a_1 R_\sigma(q)$ is satisfied for all $z=u+iv$ such that $(u,v) \in [-R,R]^2$. 
By \eqref{eq122113} and \eqref{eq122114}, we deduce from Proposition \ref{propCM1} the upper bound 
\begin{gather}\label{eq122239}
|f(u,v)-g(u,v)|
\ll (\log{q})^{-2}
\end{gather}
for $(u,v) \in [-R,R]^2$, where the implied constant depends only on $s$. 
Hence \eqref{eq122116} yields the estimate 
\begin{gather}\label{eq122117}
I_1
\ll (\log{q})^{-2} \log\log{q}. 
\end{gather}
Next, by the definition of $\hat{f}(u,v)$, we have 
\begin{align*}
\hat{f}(u,v)
&=(f(u,v)-f(u,0)-f(0,v)+1)-(f(u,0)-1)(f(0,v)-1) \\
&=\iint_{\mathbb{R}^2} (e^{ixu}-1)(e^{iyv}-1) \,d \mu(x,y)\\
&\qquad
-\iint_{\mathbb{R}^2} (e^{ixu}-1) \,d \mu(x,y)
\cdot \iint_{\mathbb{R}^2} (e^{iyv}-1) \,d \mu(x,y). 
\end{align*}
Since $e^{i\theta}-1 \ll |\theta|$ for all $\theta \in \mathbb{R}$, we obtain 
\begin{align*}
\hat{f}(u,v)
&\ll |uv| \iint_{\mathbb{R}^2} |xy| \,d \mu(x,y)
+|u| \iint_{\mathbb{R}^2} |x| \,d \mu(x,y)
\cdot |v| \iint_{\mathbb{R}^2} |y| \,d\mu(x,y) \\
&\ll |uv| \iint_{\mathbb{R}^2} (x^2+y^2) \,d \mu(x,y). 
\end{align*}
Furthermore, we have 
\[
\iint_{\mathbb{R}^2} (x^2+y^2) \,d \mu(x,y)
=\frac{1}{N(q,s)^\mathcal{E}}
\sum_{f \in B_2(q) \setminus \mathcal{E}} \omega_q(f) |\log{L}(s,f)|^2
\ll 1
\]
by Corollary \ref{corUB}. 
Thus the estimate $\hat{f}(u,v) \ll |uv|$ follows. 
By a similar argument, we evaluate $\hat{g}(u,v)$ as
\[
\hat{g}(u,v)
\ll |uv| \iint_{\mathbb{R}^2} (x^2+y^2) \,d \nu(x,y)
\ll |uv|, 
\]
where the implied constant depends only on $s$ and the choice of $\Theta$. 
Therefore, the integral $I_2$ is estimated as
\begin{gather}\label{eq122229}
I_2
\ll \meas (C(r))
\ll (\log{q})^{-1}, 
\end{gather}
where $\meas(A)$ is the two-dimensional Lebesgue measure of $A$. 
By \eqref{eq122117} and \eqref{eq122229}, we obtain
\[
\iint_{[-R,R]^2} \left| \frac{\hat{f}(u,v)-\hat{g}(u,v)}{uv} \right| \,dudv
\ll (\log{q})^{-1}. 
\]
Then, we proceed to the second integral of the right-hand side of \eqref{eq121643}. 
Here we divide it as 
\[
\int_{-R}^{R} \left|\frac{f(u,0)-g(u,0)}{u}\right| \,du
=\int_{[-R,-r) \cup (r,R]}+ \int_{[-r,r]}
=I_3+I_4,  
\]
say. 
By upper bound \eqref{eq122239}, the integral $I_3$ is estimated as
\[
I_3
\ll \left(\log\frac{R}{r}\right)
\sup_{u \in [-R,R]} |f(u,0)-g(u,0)|
\ll (\log{q})^{-2} \log\log{q}. 
\]
Furthermore, we have 
\begin{align*}
f(u,0)-g(u,0)
&=\iint_{\mathbb{R}^2} (e^{ixu}-1) \,d \mu(x,y)
-\iint_{\mathbb{R}^2} (e^{iyv}-1) \,d\nu(x,y) \\
&\ll |u| \left( \iint_{\mathbb{R}^2} x^2 \,d \mu(x,y) \right)^{1/2}
+|u| \left( \iint_{\mathbb{R}^2} x^2 \,d\nu (x,y) \right)^{1/2}\\
&\ll_{\sigma,\Theta}|u|, 
\end{align*}
and therefore, we deduce $I_4 \ll r = (\log{q})^{-2}$. 
The third integral of the right-hand side of \eqref{eq121643} is estimated along the same line. 
From the above, we finally arrive at
\[
\sup_{(x,y) \in \mathbb{R}^2} |F(x,y)-G(x,y)|
\ll (\log{q})^{-1} + R_\sigma(q)^{-1}
\ll R_\sigma(q)^{-1} 
\]
by Lemma \ref{lemEsseen2}. 
Note that the inequality 
\begin{align*}
|\mu(R)-\nu(R)|
&\leq|F(b,d)-G(b,d)|-|F(a,d)-G(a,d)| \\
&\qquad
-|F(b,c)-G(b,c)|+|F(a,c)-G(a,c)|
\end{align*}
holds if we write $R=[a,b] \times [c,d]$. 
Hence 
\begin{gather}\label{eq122311}
\sup_{ R \subset \mathbb{C}}
|\mu(R)-\nu(R)|
\ll R_\sigma(q)^{-1}
\end{gather}
follows. 
In addition, we see that 
\begin{align}\label{eq122312}
&\sup_{R \subset \mathbb{C}}
\Bigg| \sum_{f \in B'_2(q,s)} \omega_q(f)
1_R(\log{L}(s,f)) - \mu(R) \Bigg|\\
&\leq \sum_{f \in \mathcal{E}} \omega_q(f)
+\frac{|N(q,s)^\mathcal{E}-1|}{N(q,s)^\mathcal{E}}
\sum_{f \in B_2(q)} \omega_q(f)
\ll \exp\left(-\frac{b_1}{2} \frac{\log{q}}{\log\log{q}}\right) \nonumber
\end{align}
holds. 
By \eqref{eq122311} and \eqref{eq122312}, we obtain the desired discrepancy bound by recalling the equality
\[
\nu(R)
=\int_{R} \mathcal{M}_s(w,\Theta) \,|dw|, 
\]
which is deduced from Theorem \ref{thmPDF}.
\end{proof}

\begin{proof}[Proof of Theorem \ref{thmDB} $(\mathrm{ii})$]
In this case, we apply Lemma \ref{lemEsseen1} with the probability measures $\mu=\mu_{\sigma,q}^\mathcal{E}$ and $\nu=\mu_{\sigma}$, which are defined as 
\[
\mu_{\sigma,q}^\mathcal{E}(A;\omega_q)
=\frac{1}{N(q,\sigma)^\mathcal{E}}
\sum_{f \in B_2(q) \setminus \mathcal{E}} \omega_q(f)
1_A(\log{L}(\sigma,f)) 
\]
for $A \in \mathcal{B}(\mathbb{R})$ and as \eqref{eq111345}, respectively.   
Then the integral of the right-hand side of \eqref{eq122329} is evaluated as
\[
\int_{-R}^{R} \left|\frac{f(u)-g(u)}{u}\right| \,du
\ll (\log{q})^{-2} \log\log{q}
\]
along the argument for $I_3$ and $I_4$ in the proof of Theorem \ref{thmDB} $(\mathrm{i})$. 
Hence we deduce from Lemma \ref{lemEsseen1} the upper bound
\[
\sup_{x \in \mathbb{R}} |F(x)-G(x)|
\ll R_\sigma(q)^{-1}, 
\]
which yields the conclusion. 
\end{proof}

\section*{Acknowledgements}
The author would like to thank Kohji Matsumoto for providing information related to Lemma \ref{lemIhaMat} and making the latest version of his preprint \cite{Matsumoto2021+} available.


\begin{thebibliography}{10}

\bibitem{Barban1966}
M.~B. Barban, \emph{The ``large sieve'' method and its application to number
  theory}, Uspehi Mat. Nauk \textbf{21} (1966), no.~1, 51--102. \MR{0199171}

\bibitem{BohrJessen1930}
H.~Bohr and B.~Jessen, \emph{\"{U}ber die {W}erteverteilung der {R}iemannschen
  {Z}etafunktion}, Acta Math. \textbf{54} (1930), no.~1, 1--35. \MR{1555301}

\bibitem{BohrJessen1932}
\bysame, \emph{\"{U}ber die {W}erteverteilung der {R}iemannschen
  {Z}etafunktion}, Acta Math. \textbf{58} (1932), no.~1, 1--55. \MR{1555343}

\bibitem{BorchseniusJessen1948}
V.~Borchsenius and B.~Jessen, \emph{Mean motions and values of the {R}iemann
  zeta function}, Acta Math. \textbf{80} (1948), 97--166. \MR{27796}

\bibitem{Brumer1995}
A.~Brumer, \emph{The rank of {$J_0(N)$}}, no. 228, 1995, Columbia University
  Number Theory Seminar (New York, 1992), pp.~41--68. \MR{1330927}

\bibitem{ChowlaErdos1951}
S.~Chowla and P.~Erd\"{o}s, \emph{A theorem on the distribution of the values
  of {$L$}-functions}, J. Indian Math. Soc. (N.S.) \textbf{15} (1951), 11--18.
  \MR{44566}

\bibitem{CogdellMichel2004}
J.~Cogdell and P.~Michel, \emph{On the complex moments of symmetric power
  {$L$}-functions at {$s=1$}}, Int. Math. Res. Not. (2004), no.~31, 1561--1617.
  \MR{2035301}

\bibitem{Elliott1970}
P.~D. T.~A. Elliott, \emph{The distribution of the quadratic class number},
  Litovsk. Mat. Sb. \textbf{10} (1970), 189--197. \MR{0285505}

\bibitem{Elliott1971}
\bysame, \emph{On the distribution of the values of {D}irichlet {$L$}-series in
  the half-plane {$\sigma>\frac{1}{2}$}}, Nederl. Akad. Wetensch. Proc. Ser. A
  {\bf 74}=Indag. Math. \textbf{33} (1971), 222--234. \MR{0291100}

\bibitem{Elliott1972}
\bysame, \emph{On the distribution of {${\rm arg}L(s,\chi )$} in the half-plane
  {$\sigma >\frac{1}{2}$}}, Acta Arith. \textbf{20} (1972), 155--169.
  \MR{323734}

\bibitem{Elliott1973}
\bysame, \emph{On the distribution of the values of quadratic {$L$}-series in
  the half-plane {$\sigma >\frac{1}{2}$}}, Invent. Math. \textbf{21} (1973),
  319--338. \MR{352019}

\bibitem{Fomenko2003}
O.~M. Fomenko, \emph{On the distribution of the values of {$L(1,f)$}}, Zap.
  Nauchn. Sem. S.-Peterburg. Otdel. Mat. Inst. Steklov. (POMI) \textbf{302}
  (2003), no.~Anal. Teor. Chisel i Teor. Funkts. 19, 135--148. \MR{2023037}

\bibitem{Fomenko2005}
\bysame, \emph{On the distribution of the values of {$L(1,{\rm sym}^2f)$}},
  Algebra i Analiz \textbf{17} (2005), no.~6, 184--206. \MR{2202450}

\bibitem{Golubeva2004}
E.~P. Golubeva, \emph{Distribution of the values of {H}ecke {$L$}-functions at
  the point 1}, Zap. Nauchn. Sem. S.-Peterburg. Otdel. Mat. Inst. Steklov.
  (POMI) \textbf{314} (2004), no.~Anal. Teor. Chisel i Teor. Funkts. 20,
  15--32. \MR{2119731}

\bibitem{HarmanMatsumoto1994}
G.~Harman and K.~Matsumoto, \emph{Discrepancy estimates for the
  value-distribution of the {R}iemann zeta-function. {IV}}, J. London Math.
  Soc. (2) \textbf{50} (1994), no.~1, 17--24. \MR{1277751}

\bibitem{IharaMatsumoto2011a}
Y.~Ihara and K.~Matsumoto, \emph{On certain mean values and the
  value-distribution of logarithms of {D}irichlet {$L$}-functions}, Q. J. Math.
  \textbf{62} (2011), no.~3, 637--677. \MR{2825476}

\bibitem{IharaMatsumoto2011b}
\bysame, \emph{On {$\log L$} and {$L'/L$} for {$L$}-functions and the
  associated ``{$M$}-functions'': connections in optimal cases}, Mosc. Math. J.
  \textbf{11} (2011), no.~1, 73--111. \MR{2808212}

\bibitem{IharaMurtyShimura2009}
Y.~Ihara, V.~K. Murty, and M.~Shimura, \emph{On the logarithmic derivatives of
  {D}irichlet {$L$}-functions at {$s=1$}}, Acta Arith. \textbf{137} (2009),
  no.~3, 253--276. \MR{2496464}

\bibitem{IwaniecKowalski2004}
H.~Iwaniec and E.~Kowalski, \emph{Analytic number theory}, American
  Mathematical Society Colloquium Publications, vol.~53, American Mathematical
  Society, Providence, RI, 2004. \MR{2061214}

\bibitem{JessenWintner1935}
B.~Jessen and A.~Wintner, \emph{Distribution functions and the {R}iemann zeta
  function}, Trans. Amer. Math. Soc. \textbf{38} (1935), no.~1, 48--88.
  \MR{1501802}

\bibitem{KowalskiMichel1999}
E.~Kowalski and P.~Michel, \emph{The analytic rank of {$J_0(q)$} and zeros of
  automorphic {$L$}-functions}, Duke Math. J. \textbf{100} (1999), no.~3,
  503--542. \MR{1719730}

\bibitem{Lamzouri2011b}
Y.~Lamzouri, \emph{On the distribution of extreme values of zeta and
  {$L$}-functions in the strip {$\frac{1}{2}<\sigma<1$}}, Int. Math. Res. Not.
  IMRN (2011), no.~23, 5449--5503. \MR{2855075}

\bibitem{LamzouriLesterRadziwill2019}
Y.~Lamzouri, S.~J. Lester, and M.~Radziwi{\l\l}, \emph{Discrepancy bounds for
  the distribution of the {R}iemann zeta-function and applications}, J. Anal.
  Math. \textbf{139} (2019), no.~2, 453--494. \MR{4041109}

\bibitem{LebacqueZykin2018}
P.~Lebacque and A.~Zykin, \emph{On {$M$}-functions associated with modular
  forms}, Mosc. Math. J. \textbf{18} (2018), no.~3, 437--472. \MR{3860846}

\bibitem{Loeve1977}
M.~Lo\`eve, \emph{Probability theory. {I}}, fourth ed., Springer-Verlag, New
  York-Heidelberg, 1977, Graduate Texts in Mathematics, Vol. 45. \MR{0651017}

\bibitem{Luo1999}
W.~Luo, \emph{Values of symmetric square {$L$}-functions at {$1$}}, J. Reine
  Angew. Math. \textbf{506} (1999), 215--235. \MR{1665705}

\bibitem{Matsumoto1989}
K.~Matsumoto, \emph{A probabilistic study on the value-distribution of
  {D}irichlet series attached to certain cusp forms}, Nagoya Math. J.
  \textbf{116} (1989), 123--138. \MR{1029974}

\bibitem{Matsumoto1990}
\bysame, \emph{Value-distribution of zeta-functions}, Analytic number theory
  ({T}okyo, 1988), Lecture Notes in Math., vol. 1434, Springer, Berlin, 1990,
  pp.~178--187. \MR{1071754}

\bibitem{Matsumoto1991}
\bysame, \emph{On the magnitude of asymptotic probability measures of
  {D}edekind zeta-functions and other {E}uler products}, Acta Arith.
  \textbf{60} (1991), no.~2, 125--147. \MR{1139051}

\bibitem{Matsumoto1992}
\bysame, \emph{Asymptotic probability measures of zeta-functions of algebraic
  number fields}, J. Number Theory \textbf{40} (1992), no.~2, 187--210.
  \MR{1149737}

\bibitem{Matsumoto2021+}
\bysame, \emph{An {$M$}-function associated with {G}oldbach's problem}, to 
  appear in J. Ramanujan Math. Soc.

\bibitem{MatsumotoUmegaki2018}
K.~Matsumoto and Y.~Umegaki, \emph{On the value-distribution of the difference
  between logarithms of two symmetric power {$L$}-functions}, Int. J. Number
  Theory \textbf{14} (2018), no.~7, 2045--2081. \MR{3831410}

\bibitem{Royer2001}
E.~Royer, \emph{Statistique de la variable al\'{e}atoire {$L({\rm
  sym}^2f,1)$}}, Math. Ann. \textbf{321} (2001), no.~3, 667--687. \MR{1871974}

\bibitem{RudnickSoundararajan2006}
Z.~Rudnick and K.~Soundararajan, \emph{Lower bounds for moments of
  {$L$}-functions: symplectic and orthogonal examples}, Multiple {D}irichlet
  series, automorphic forms, and analytic number theory, Proc. Sympos. Pure
  Math., vol.~75, Amer. Math. Soc., Providence, RI, 2006, pp.~293--303.
  \MR{2279944}

\bibitem{Sadikova1966}
S.~M. Sadikova, \emph{On two-dimensional analogs of an inequality of {E}sseen
  and their application to the central limit theorem}, Teor. Verojatnost. i
  Primenen. \textbf{11} (1966), 369--380. \MR{0207016}

\bibitem{Stankus1975a}
E.~Stankus, \emph{The distribution of {D}irichlet {$L$}-functions}, Litovsk.
  Mat. Sb. \textbf{15} (1975), no.~2, 127--134. \MR{0384718}

\bibitem{Stankus1975b}
\bysame, \emph{Distribution of {D}irichlet {$L$}-functions with real characters
  in the half-plane {${\rm Re}$} {$s>1/2$}}, Litovsk. Mat. Sb. \textbf{15}
  (1975), no.~4, 199--214. \MR{0406956}

\end{thebibliography}

\providecommand{\bysame}{\leavevmode\hbox to3em{\hrulefill}\thinspace}
\providecommand{\MR}{\relax\ifhmode\unskip\space\fi MR }
\providecommand{\MRhref}[2]{%
  \href{http://www.ams.org/mathscinet-getitem?mr=#1}{#2}
}
\providecommand{\href}[2]{#2}

\end{document}